\documentclass[12pt,a4paper]{amsart}
\usepackage{amsfonts}
\usepackage{amsthm}
\usepackage{amsmath}
\usepackage{amscd}
\usepackage{amssymb} % for varnothing
\usepackage[latin2]{inputenc}
\usepackage{graphicx}
\numberwithin{equation}{section}
\usepackage[margin=2.6cm]{geometry}
\usepackage{dsfont}
\usepackage{url}
\usepackage{enumitem}
\usepackage{color}
\usepackage[toc]{appendix}
\usepackage{bm}
\usepackage{mathtools} % for bigtimes

\makeatletter
\newcommand*{\rom}[1]{\expandafter\@slowromancap\romannumeral #1@}
\makeatother

\theoremstyle{plain}
\newtheorem{theorem}{Theorem}[section]
\newtheorem{Lemma}[theorem]{Lemma}
\newtheorem{Cor}[theorem]{Corollary}
\newtheorem{Prop}[theorem]{Proposition}

\theoremstyle{definition}

\newtheorem{Rem}[theorem]{Remark}
\newtheorem{?}[theorem]{Problem}
\newtheorem{exa}[theorem]{Example}

\usepackage{hyperref}
\begin{document}
	\title[Multivariate Archimedean copulas and their conditional copulas]{On convergence and singularity of conditional copulas of multivariate Archimedean copulas, and estimating conditional dependence}
	\author[]{Thimo M. Kasper}
%	\address{\ \newline Thimo M. Kasper \newline University of Salzburg \newline Department for Artificial Intelligence and Human Interfaces \newline Hellbrunnerstrasse 34, 5020 Salzburg}
	\address{Thimo M. Kasper \newline \hspace*{0.3cm} University of Salzburg \newline\hspace*{0.3cm} Department for Artificial Intelligence and Human Interfaces \newline\hspace*{0.3cm} Hellbrunnerstrasse 34, 5020 Salzburg}
	\email{thimo.kasper@plus.ac.at}
	
	\begin{abstract} 
	The present contribution derives an explicit expression for (a version of) every uni- and multivariate conditional distribution (i.e., Markov kernel) of Archimedean copulas and uses this representation to generalize a recently established result, saying that in the class of multivariate Archimedean copulas standard uniform convergence implies weak convergence of almost all univariate Markov kernels, to arbitrary multivariate Markov kernels.
	Moreover, we prove that an Archimedean copula is singular if, and only if, almost all uni- and multivariate Markov kernels are singular.
	These results are then applied to conditional Archimedean copulas which are reintroduced largely from a Markov kernel perspective and it is shown that convergence, singularity and conditional increasingness carry over from Archimedean copulas to their conditional copulas.
	As consequence the surprising fact is established that estimating (the generator of) an Archimedean copula directly yields an estimator of (the generator of) its conditional copula. Building upon that, we sketch the use and estimation of a conditional version of a recently introduced dependence measure as alternative to well-known conditional versions of association measures in order to study the dependence behaviour of Archimedean models when fixing covariate values.
	\end{abstract}
	
	\maketitle
%	\tableofcontents
%	\vspace*{-0.6cm}
%	
	
	\section{Introduction}
	\label{sec:1:intro}
	
	Archimedean copulas pose a popular class of dependence structures. As the class includes various parametric families covering a wide range of dependence models, Archimedean copulas are applied in numerous fields such as hydrology (see, e.g., \cite{archimedeanHydro}), finance (see, e.g., \cite{archimedeanFinance}) or actuarial sciences (see, e.g., \cite{archimedeanActuar}) among many others. Archimedean copulas are also attractive from a theoretical point of view as they enjoy the property of being generated by a single univariate function $\psi$ which fully characterizes analytic and dependence aspects of Archimedean copulas (\cite{convArch,Principles, p21, wcc, convMultiArch,multiArchNeslehova, Nelsen}). The current paper is particularly concerned with convergence as well as singularity of multivariate Archimedean copulas and it turns out that studying them via their conditional distributions (i.e., Markov kernels) leads to simple, alternative derivations of important results in the field (see \cite{p21,wcc,multiArchMassKernel,convMultiArch}). In the two-dimensional case it is well-known that within the class of multivariate Archimedean copulas pointwise convergence of the copulas coincides with pointwise convergence of the corresponding generators as well as with weak convergence of almost all Markov kernels (\cite{wcc}) while in \cite{convMultiArch} it was recently shown that pointwise convergence of \emph{multivariate} Archimedean copulas implies convergence of almost all univariate Markov kernels. 
	%Furthermore, the authors showed that a $d$-dimensional Archimedean copula is singular w.r.t. the $d$-dimensional Lebesgue  measure $\lambda_d$ if, and only if, almost all univariate conditional distributions are singular w.r.t. $\lambda_1$. 
	
	In this contribution we derive an explicit expression for all uni- and multivariate Markov kernels of Archimedean copulas, hence allowing for an arbitrary number of conditioning coordinates. Utilizing this representation we prove that, in addition to several well-known characterizing properties, pointwise convergence of Archimedean copulas is equivalent to weak convergence of almost all $(d-1)$-dimensional Markov kernels and that any of these properties implies weak convergence of almost all uni- and multivariate Markov kernels. Extending \cite{convMultiArch} in another aspect, we show the surprising fact that an Archimedean copula is singular w.r.t. $\lambda_d$ if, and only if, almost all uni- and multivariate conditional distributions are singular w.r.t. the appropriate Lebesgue measure. 
	
	These results are then applied to conditional Archimedean copulas. In general, conditional copulas allow for a means to study the dependence of a random vector when influenced by a (vector of) covariate(s). In \cite{gijbels_partial_copulas,gijbels_conditionalCop}, for instance, the authors study the estimation of conditional copulas from data and define conditional versions of the concordance measures Kendall's tau and Spearman's rho. Conditional copulas of Archimedean copulas, in particular, are studied, e.g., in \cite{ConditionalCopula, conditionalCop_Sungur}. In \cite{ConditionalCopula} the authors study tail behaviour, conditional versions of tail dependence indices, positive dependence and show the surprising fact that the conditional copula of an Archimedean copula is again Archimedean (also called the truncation invariance property). In the present paper we will reintroduce conditional copulas largley from a Markov kernel perspective, provide an alternative proof of the afore-mentioned truncation invariance property and demonstrate that the Markov kernel representation allows for several beautiful relations concerning conditional Archimedean generators. Using our established results we will then show that convergence, singularity and conditional increasingness carry over from the original Archimedean copula to its conditional copula. As consequence we obtain the surprising fact that for estimating the conditional copula of an Archimedean copula $C$ it suffices to estimate (the generator of) $C$ which can be done using well-known methods (see, e.g., \cite{GenestInference, estimationArch_hering, estimationArch_knownMarginals} and the references therein). Going a step further and utilizing this observation we sketch the use and estimation of a conditional version of the multivariate dependence measure $\zeta_1$ (\cite{qmd}) as a new tool to investigate dependence behaviour of Archimedean data given covariate values, thus providing a potentially powerful alternative to the well-known conditional Kendall's tau (\cite{gijbels_conditionalCop}) or general tail concentration functions (see, e.g., \cite{Durante_tail_dep_diagonal_201522}).

	The remainder of this paper is organized as follows: Section~\ref{sec:2:preliminaries} gathers notation and preliminaries that are used throughout the text. Focussing on Archimedean copulas, Section~\ref{sec:kernel:archimedean} derives an explicit expression for (a version of) every uni- and multivariate Markov kernel. This representation is then used in Section~\ref{sec:convergence:D1} to characterize pointwise/uniform convergence of Archimedean copulas in terms of their $(d-1)$-dimensional Markov kernels and to show that, in this setting, uniform convergence even implies weak convergence of almost all uni- \emph{and} multivariate Markov kernels. The main result of Section~\ref{sec:singularity} is that singularity of Archimedean copulas is equivalent to singularity of almost all uni- and multivariate Markov kernels. Following up on an introduction to conditional copulas of Archimedean copulas via Markov kernels, in Section~\ref{sec:conditional:copulas} we prove that convergence as well as singularity transfers from Archimedean copulas to their conditional copulas, and in addition, sketch the use and estimation of a conditional version of the dependence measure $\zeta_1$ to obtain additional insights into the dependence behaviour of Archimedean data given covariate values. We conclude the paper by remarking that, at last, conditional increasingness carries over to the conditional Archimedean copulas, too. 
	%\ref{sec:kernel:representation} establishes Markov kernels of general high dimensional distribution functions with absolutely continuous marginals in Section \ref{sec:kernel:representation}
	
	\section{Notation and preliminaries}
	\label{sec:2:preliminaries}
	
	This paper considers the class $\mathcal{C}^d$ of all $d$-dimensional copulas for some integer $d\geq2$. Given $C\in\mathcal{C}^d$ the corresponding $d$-stochastic measure will be denoted by $\mu_C$, i.e., 
	$\mu_C([\mathbf{0},\mathbf{x}]) = C(\mathbf{x})$ for all $\mathbf{x} \in \mathbb{I}^d$, where $\mathbb{I} := [0,1]$ and $[\mathbf{0},\mathbf{x}] := [0,x_1] \times \ldots\times [0,x_d]$. 
	Considering $1\leq i < j \leq d$, the $i$-$j$-marginal $C^{ij}$ of $C$ is given by $C^{ij}(x_i, x_j) = C(1,\ldots, 1,x_i,1,\ldots,1,x_j,1,\ldots,1)$ and similarly for $1<\ell<d$ the marginal copula $C^{1:\ell}$ of the first $\ell$ coordinates we have $C^{1:\ell}(\mathbf{x}) = C(x_1,x_2,\ldots, x_\ell, 1,\ldots, 1)$ for every $\mathbf{x}\in\mathbb{I}^\ell$. The uniform metric on $\mathcal{C}^d$ is defined by
	\begin{align*}
		d_\infty(C_1,C_2) := \max_{\mathbf{x}\in\mathbb{I}^d}\vert C_1(\mathbf{x}) - C_2(\mathbf{x})\vert
	\end{align*} 
	and it is well-known that $(\mathcal{C}^d, d_\infty)$ is a compact metric space. For more background on multivariate copulas and $d$-stochastic probability measures we refer to \cite{Nelsen, Principles}. 
	
	For every metric space $(\mathbb{S},d_\mathbb{S})$ the Borel $\sigma$-field on $\mathbb{S}$ will be denoted by $\mathcal{B}(\mathbb{S})$ and we write $\lambda_{d}$ for the Lebesgue measure on $\mathcal{B}(\mathbb{I}^d)$. In the sequel Markov kernels will play the key role. For $1<\ell<d$ an $\ell$-\emph{Markov kernel} from $\mathbb{R}^\ell$ to $\mathbb{R}^{d-\ell}$ is a mapping $K: \mathbb{R}^\ell\times\mathcal{B}(\mathbb{R}^{d-\ell}) \rightarrow \mathbb{I}$ such that for every fixed $E\in\mathcal{B}(\mathbb{R}^{d-\ell})$ the mapping $\mathbf{x}\mapsto K(\mathbf{x},E)$ is $\mathcal{B}(\mathbb{R}^{\ell})$-$\mathcal{B}(\mathbb{R}^{d-\ell})$-measurable and for every fixed $\mathbf{x}\in\mathbb{R}^\ell$ the mapping $E\mapsto K(\mathbf{x},E)$ is a probability measure on $\mathcal{B}(\mathbb{R}^{d-\ell})$. Given a $(d-\ell)$-dimensional random vector $\mathbf{Y}$ and a
	$\ell$-dimensional random vector $\mathbf{X}$ on a probability space
	$(\Omega, \mathcal{A}, \mathbb{P})$
	we say that a Markov kernel $K$ is a regular conditional distribution of
	$\mathbf{Y}$ given $\mathbf{X}$ if $K(\mathbf{X}(\omega), E) = \mathbb{E}(\mathbf{1}_E \circ \mathbf{Y} | \mathbf{X})(\omega)$ holds $\mathbb{P}$-almost surely for every $E \in
	\mathcal{B}(\mathbb{R}^{d-\ell})$. It is well-known that for each pair of random vectors $(\mathbf{X}, \mathbf{Y})$ as above, a regular conditional distribution $K(\cdot, \cdot)$ of $\mathbf{Y}$ given $\mathbf{X}$ always exists and is unique for $\mathbb{P}^{\mathbf{X}}$-a.e. $\mathbf{x} \in \mathbb{R}^\ell$
	whereby $\mathbb{P}^{\mathbf{X}}$ denotes the push-forward of $\mathbb{P}$ via $\mathbf{X}$.
	In case $(\mathbf{X}, \mathbf{Y})$ has distribution function $H$ we let $K_{H}:\mathbb{R}^\ell \times \mathcal{B}(\mathbb{R}^{d-\ell}) \to \mathbb{I}$ denote (a version of) the regular conditional distribution of $\mathbf{Y}$ given $\mathbf{X}$ and refer to it as $\ell$-\emph{Markov kernel} of $H$.
	Defining the $\mathbf{x}$-section of a set $G \in\mathcal{B}(\mathbb{R}^d)$ w.r.t. the first $\ell$ coordinates by $G_{\mathbf{x}}:=\{\mathbf{y}
	\in \mathbb{R}^{d-\ell}: (\mathbf{x},\mathbf{y}) \in G\}\in\mathcal{B}(\mathbb{R}^{d-\ell})$ the so-called \emph{disintegration theorem} implies
	\begin{align*}
		\mu_H(G) = \int_{\mathbb{R}^\ell} K_{H}(\mathbf{x},G_{\mathbf{x}})
		\ \mathrm{d}\mu_{H^{1:\ell}}(\mathbf{x}).
	\end{align*}
	For more background on conditional expectation and disintegration we refer to \cite{Kallenberg} and \cite{Klenke}. 
	
	Markov kernels allow to define metrics stronger than $d_\infty$. Indeed, following \cite{p06} and defining for two copulas $A,B\in\mathcal{C}^d$
	\begin{align*}
		D_1(A,B) &:= \int_{\mathbb{I}^d}|K_A(x,[\mathbf{0},\mathbf{y}]) - K_B(x,[\mathbf{0},\mathbf{y}])| \ \mathrm{d}\lambda_d(x,\mathbf{y}), \\
		D_2^2(A,B) &:= \int_{\mathbb{I}^d}(K_A(x,[\mathbf{0},\mathbf{y}]) - K_B(x,[\mathbf{0},\mathbf{y}]))^2 \ \mathrm{d}\lambda_d(x,\mathbf{y}), \\
		D_\infty(A,B) &:= \sup_{\mathbf{y}\in\mathbb{I}^{d-1}} \int_{\mathbb{I}} |K_A(x,[\mathbf{0},\mathbf{y}]) - K_B(x,[\mathbf{0},\mathbf{y}])| \ \mathrm{d}\lambda_1(x)
	\end{align*}
	yield $1$-Markov kernel based metrics on $\mathcal{C}^d$ and it is well-known that convergence w.r.t. either of these metrics implies uniform convergence (and not vice versa).
	
	In this paper we will mainly consider the class $\mathcal{C}^d_\text{ar}$ of Archimedean copulas. Following \cite{multiArchNeslehova}, $C\in\mathcal{C}^d_\text{ar}$ if there exists a continuous, non-increasing function $\psi:[0,\infty)\to\mathbb{I}$ fulfilling $\psi(0)=1,\lim_{z\to\infty}\psi(z) = 0$ and being strictly decreasing on $[0,\inf\{z\in[0,\infty]: \psi(z) = 0\}]$ (where we set $\psi(\infty) = 0$) such that for every $\mathbf{x}\in\mathbb{I}^d$ we have
	\begin{align*}
		C(\mathbf{x}) = \psi(\varphi(x_1) + \cdots + \varphi(x_d)),
	\end{align*}
	where $\varphi:(0,1]\to[0,\infty)$ denotes the pseudoinverse of $\psi$. Setting $\varphi(0) = \inf\{z\in[0,\infty]: \psi(z) = 0\}$ the function $\varphi$ is strictly decreasing and satisfies $\varphi(1) = 0$. In case of $\varphi(0) = \infty$ we call $C$ strict and non-strict otherwise. Conversely, the function $\psi(\varphi(x_1) + \cdots + \varphi(x_d))$ for $\mathbf{x}\in\mathbb{I}^d$ returns a $d$-dimensional copula if, and only if, $\psi$ is $d$-monotone on $[0,\infty)$ meaning that $\psi$ is continuous on $[0,\infty)$, has derivatives up to order $(d-2)$ on $(0,\infty)$, fulfills $(-1)^m\psi^{(m)}(z) \geq 0$ for every $z\in(0,\infty)$, $m=0,1,\ldots,d-2$, and it holds that $(-1)^{d-2}\psi^{(d-2)}$ is non-increasing and convex on $(0,\infty)$. Thereby $\psi^{(m)}$ denotes the $m$-th derivative of $\psi$. Letting $D^-\psi$ ($D^+\psi$) denote the left (right) derivative of $\psi$ and considering convexity of $(-1)^{d-2}\psi^{(d-2)}$ we have $D^-\psi = D^+\psi$ almost everywhere on $(0,\infty)$. As direct consequence the set of discontinuity points $(0,\infty)\setminus$ Cont$(D^-\psi^{(d-2)})$ of $D^-\psi^{(d-2)}$ is at most countably infinite and it follows directly from the properties of $\psi$ that $\lim_{z\to\infty} D^-\psi^{(d-2)}(z) = 0$ holds (see \cite{convMultiArch}). As generators of Archimedean copulas are not unique, from now on we will assume that generators are normalized in the sense that $\varphi(\frac{1}{2}) = 1$, or equivalently, $\psi(1) = \frac{1}{2}$. 
	
	\section{Markov kernels of multivariate Archimedean copulas}
	\label{sec:kernel:archimedean}
	It is well-known (cf. \cite{multiArchNeslehova}) that in case of $C\in\mathcal{C}_\text{ar}^d$ being absolutely continuous (a version of) its density is given by
	\begin{align}\label{eq:arch:density}
		c(\mathbf{x}) = \mathds{1}_{(0,1)^d}(\mathbf{x}) \prod_{i=1}^{d}\varphi'(x_i) \cdot D^-\psi^{(d-1)}\big( \varphi(x_1)+\cdots +\varphi(x_d) \big)
	\end{align} 
	and that all lower dimensional marginals of Archimedean copulas are absolutely continuous. Combining the latter property with Theorem \ref{th:interkernel}, Lemma \ref{lem:partial:increasing} in \ref{sec:kernel:representation} and \cite[Theorem 3.1]{convMultiArch} allows to derive an explicit expression for (a version of) \emph{every} Markov kernel of an Archimedean copula. To that end, define the zero set $L_0$ of $C\in\mathcal{C}^d_\text{ar}$ as
	\begin{align*}
		L_0 &= \{(\mathbf{x},\mathbf{y})\in\mathbb{I}^\ell\times\mathbb{I}^{d-\ell}: C(\mathbf{x},\mathbf{y}) = 0\} \\ &=\left\{(\mathbf{x},\mathbf{y})\in\mathbb{I}^\ell\times\mathbb{I}^{d-\ell}: \sum_{i=1}^{\ell}\varphi(x_i) + \sum_{j=1}^{d-\ell}\varphi(y_i) \geq \varphi(0)\right\}.
	\end{align*}
	In what follows we will work with the zero level set of the $\ell$-dimensional marginal of $C$ which is defined analogously and denoted by $L_0^{1:\ell}$. Writing int$S$ for the interior of a set $S$, notice that $(\mathbf{x},\mathbf{y})\in$ int$L_0$ if, and only if, $\sum_{i=1}^{\ell}\varphi(x_i) + \sum_{j=1}^{d-\ell}\varphi(y_j) > \varphi(0)$ and that for strict $C\in\mathcal{C}^d_\text{ar}$ we have $\mathbf{x}\in L_0^{1:\ell}$ if, and only if, $\min(\mathbf{x}) = 0$
	
	\begin{theorem}\label{th:interkernel:arch}
		Suppose that $C_\psi$ is a $d$-dimensional Archimedean copula with generator $\psi$. Then for $\ell\in\{1,\ldots,d-1\}$ setting
		\begin{align}\label{eq:interkernel:arch}
			K_\psi(\mathbf{x}, [\mathbf{0},\mathbf{y}]) = \begin{cases}
				1 & \emph{ if } \min(\mathbf{x}) =1 \emph{ or } \mathbf{x}\in L_0^{1:\ell} \\
				0 & \emph{ if } \min(\mathbf{x})\in (0,1), \mathbf{x}\not\in L_0^{1:\ell}, (\mathbf{x}, \mathbf{y})\in\emph{int}L_0 \\
				\frac{\psi^{(\ell)}\left( \sum_{i=1}^\ell \varphi(x_i) + \sum_{j=1}^{d-\ell}\varphi(y_i) \right)}{\psi^{(\ell)}\left( \sum_{i=1}^\ell \varphi(x_i) \right)} & \emph{ if } \min(\mathbf{x})\in (0,1), \mathbf{x}\not\in L_0^{1:\ell}, (\mathbf{x}, \mathbf{y})\not\in\emph{int}L_0
			\end{cases}
		\end{align}
		yields (a version of) the $\ell$-Markov kernel of $C_\psi$.
	\end{theorem}
	\begin{proof}
		For $\min(\mathbf{x})\in\{0,1\}$ or $\mathbf{x}\in$ $L_0^{1:\ell}$ the mapping $F_\mathbf{x}:$ $\mathbf{y}\mapsto K_\psi(\mathbf{x},[\mathbf{0},\mathbf{y}])$ is obviously a $(d-\ell)$-dimensional distribution function so fix $\mathbf{x}\not\in$ $L_0^{1:\ell}, \min(\mathbf{x})\in(0,1)$. First notice that right-continuity of $F_\mathbf{x}$ in every coordinate is a consequence of the continuity of $\psi^{(\ell)}$ (if $\ell=d-1$, $\psi^{(\ell)}$ is replaced be the left-hand derivative $D^-\psi^{(d-2)}$ whose left-continuity yields right-continuity of $F_\mathbf{x}$). Concerning $(d-\ell)$-increasingness suppose  $\mathbf{a},\mathbf{b}\in\mathbb{I}^{d-\ell}$ such that $b_j\geq a_j$ for $j=1,\ldots, d-\ell$. As in the previous lemma we have to verify that $V_{F_\mathbf{x}}((\mathbf{a},\mathbf{b}]) \geq 0$ holds.
		We distinguish the following cases:
		(i) $(\mathbf{x},\mathbf{a}), (\mathbf{x}, \mathbf{b}) \not\in$ int$L_0$. Then $F_\mathbf{x}(\mathbf{y}) = \frac{\partial^\ell C(\mathbf{x},\mathbf{y})}{\partial x_1\partial x_2\cdots \partial x_\ell} \cdot c^{1:\ell}(\mathbf{x})^{-1}$ and Lemma \ref{lem:partial:increasing} yields $(d-\ell)$-increasingness. (ii) $(\mathbf{x},\mathbf{b})\in$ int$L_0$. Then also $(\mathbf{x},\mathbf{a})\in$ int$L_0$ so that $V_{F_\mathbf{x}}((\mathbf{a},\mathbf{b}]) = 0$ holds. (iii) $(\mathbf{x},\mathbf{a})\in$ int$L_0$ and $(\mathbf{x},\mathbf{b})\not\in$ int$L_0$. Whenever a vertex $\mathbf{v}$ of $(\mathbf{a},\mathbf{b}]$ fulfills $\mathbf{v}\in$ int$L_0$ we have $\sum_{i=1}^{\ell}\varphi(x_i) + \sum_{j=1}^{d-\ell}\varphi(v_j) > \varphi(0)$ and hence $F_\mathbf{x}(\mathbf{v}) = 0$. 
		%		Hence $\sum_{i=1}^\ell \varphi(x_i) + \sum_{j=1}^{d-\ell}\varphi(a_j) > \varphi(0)$ so there exists $r\in\{1,\ldots,d-\ell\}$ such that $\sum_{i=1}^\ell \varphi(x_i) + \sum_{j=1}^{r}\varphi(a_j) > \varphi(0)$ (w.l.o.g.  the first $r$ coordinates of $\mathbf{a}$). Consequently, whenever $a_1,a_2,\ldots,a_r$ are part of a vertex $\mathbf{v}$ of $(\mathbf{a},\mathbf{b}]$ we have $F_\mathbf{x}(\mathbf{v}) = 0$. 
		Moreover, we already know that $\psi^{(\ell)}\left(\sum_{i=1}^\ell \varphi(x_i) + \sum_{j=1}^{d-\ell}\varphi(v_j)\right) = 0$ holds too, so again we have
		$$ \setlength{\abovedisplayskip}{7pt}
			\setlength{\belowdisplayskip}{7pt}
			F_\mathbf{x}(\mathbf{v}) = 0 = \frac{\partial^\ell C(\mathbf{x},\mathbf{v})}{\partial x_1\partial x_2\cdots \partial x_\ell} \cdot c^{1:\ell}(\mathbf{x})^{-1}
		$$
		and Lemma \ref{lem:partial:increasing} yields the desired property.
	
		Altogether, for every fixed $\mathbf{x}\in\mathbb{I}^\ell$ the map $F_\mathbf{x}: \mathbf{y}\mapsto K_\psi(\mathbf{x},[\mathbf{0},\mathbf{y}])$ is a $(d-\ell)$-dimensional distribution function and it follows that $K_\psi(\mathbf{x},\cdot)$ is a probability measure on $\mathcal{B}(\mathbb{I}^{d-\ell})$.
		%		As the collection $\{ [\mathbf{0},\mathbf{y}]: \mathbf{y}\in\mathbb{I}^{d-\ell} \}$ is a semiring, we finally extend $K_\psi(\mathbf{x},\cdot)$ to a probability measure on $\mathcal{B}(\mathbb{I}^{d-\ell})$.
		On the other hand, for every fixed $\mathbf{y}\in\mathbb{I}^{d-\ell}$ the mapping $\mathbf{x}\mapsto K_\psi(\mathbf{x},[\mathbf{0},\mathbf{y}])$ is measurable. Since the collection $\{ F\in\mathcal{B}(\mathbb{I}^{d-\ell}): \mathbf{x}\mapsto K_\psi(\mathbf{x},F) \text{ is measurable} \}$ is a Dynkin system and contains a generating class of $\mathcal{B}(\mathbb{I}^{d-\ell})$ it follows that $\mathbf{x}\mapsto K_\psi(\mathbf{x},E)$ is measurable for every $E\in\mathcal{B}(\mathbb{I}^{d-\ell})$. Consequently, $K_\psi(\cdot, \cdot)$ is a Markov kernel from $\mathbb{I}^\ell$ to $\mathbb{I}^{d-\ell}$ and it remains to show that $K_\psi(\cdot, \cdot)$ is an $\ell$-Markov kernel of $C_\psi$ which can be done in a similar manner as in \cite[Theorem 3.1]{convMultiArch} (included hereafter for the sake of completeness): We have to prove the identity
		$$
			\setlength{\abovedisplayskip}{7pt}
			\setlength{\belowdisplayskip}{7pt}
			\int_{[\mathbf{0},\mathbf{x}]} K_\psi(\mathbf{s},[\mathbf{0},\mathbf{y}]) \ \mathrm{d}\mu_{C^{1:\ell}_\psi}(\mathbf{s}) = C_\psi(\mathbf{x},\mathbf{y})
		$$
		for every $(\mathbf{x},\mathbf{y})\in\mathbb{I}^\ell\times\mathbb{I}^{d-\ell}$.
		If $\min(\mathbf{y}) = 0$ we have $(\mathbf{x},\mathbf{y})\in L_0$ whereas $\min(\mathbf{y}) = 1$ gives $
			\int_{[\mathbf{0},\mathbf{x}]} K_\psi(\mathbf{s},[\mathbf{0},\mathbf{1}]) \ \mathrm{d}\mu_{C^{1:\ell}_\psi}(\mathbf{s}) = \mu_{C^{1:\ell}}([\mathbf{0},\mathbf{x}]) = C_\psi(\mathbf{x},\mathbf{1})
		$
		so fix $\mathbf{y}\in(\mathbf{0},\mathbf{1})$. As $\mu_{C^{1:\ell}_\psi}(L_0^{1:\ell}) = 0$ and considering absolute continuity of $\mu_{C^{1:\ell}}$ we get
		\begin{align*}
			(\star) &:= \int_{[\mathbf{0},\mathbf{x}]} K_\psi(\mathbf{s},[\mathbf{0},\mathbf{y}]) \ \mathrm{d}\mu_{C^{1:\ell}_\psi}(\mathbf{s}) = \int_{[\mathbf{0},\mathbf{x}]\setminus L_0^{1:\ell}} K_\psi(\mathbf{s},[\mathbf{0},\mathbf{y}]) c^{1:\ell}(\mathbf{s}) \ \mathrm{d}\lambda_\ell(\mathbf{s}) \\
			&= \int_{[\mathbf{0},\mathbf{x}]\cap \{\mathbf{t}\in\mathbb{I}^\ell\setminus L_0^{1:\ell}: (\mathbf{t},\mathbf{y})\not\in\text{int}L_0\}} \prod_{i=1}^{\ell}\varphi'(s_i)\psi^{(\ell)}\left(\sum_{i=1}^{\ell}\varphi(s_i)+\sum_{j=1}^{d-\ell}\varphi(y_j)\right) \ \mathrm{d}\lambda_\ell(\mathbf{s}).
		\end{align*}
		If $\mathbf{s}\in L_0^{1:\ell}$ then $(\mathbf{s},\mathbf{y})\in L_0$ if, and only if, $\sum_{i=1}^{\ell}\varphi(x_i) + \sum_{j=1}^{d-\ell}\varphi(y_j) > \varphi(0)$ and hence $\psi^{(\ell)}\left(\sum_{i=1}^{\ell}\varphi(x_i) + \sum_{j=1}^{d-\ell}\varphi(y_j)\right) = 0$ and whenever $\mathbf{s}\not\in L_0^{1:\ell}$ we have $\psi^{(\ell)}\left(\sum_{i=1}^{\ell}\varphi(x_i) + \sum_{j=1}^{d-\ell}\varphi(y_j)\right) = 0$, too. As direct consequence,
		\begin{align*}
			(\star) &= \int_{(\mathbf{0},\mathbf{x}]} \prod_{i=1}^{\ell}\varphi'(s_i)\psi^{(\ell)}\left(\sum_{i=1}^{\ell}\varphi(s_i)+\sum_{j=1}^{d-\ell}\varphi(y_j)\right) \ \mathrm{d}\lambda_\ell(\mathbf{s}) \\
			&= \int_{(\mathbf{0},\mathbf{x}_{1:\ell-1}]} \prod_{i=1}^{\ell-1}\varphi'(s_i)
			\int_{(0,x_\ell]} \varphi'(s_\ell) 			\psi^{(\ell)}\left(\sum_{i=1}^{\ell}\varphi(s_i)+\sum_{j=1}^{d-\ell}\varphi(y_j)\right) \ \mathrm{d}\lambda(s_\ell)\mathrm{d}\lambda_{\ell-1}(\mathbf{s}_{1:\ell-1})
		\end{align*}
		Applying change of coordinates and considering that $\psi^{(r)}(\infty) := \psi^{(r)}(\lim_{z\to\infty}z) = 0$ for every $r\in\{1,\ldots,d-2\}$ yields
		\begin{align*}
			(\star) &= \int_{(\mathbf{0},\mathbf{x}_{1:\ell-1}]} \prod_{i=1}^{\ell}\varphi'(s_i) \lim_{\delta\to 0}\Biggl( \psi^{(\ell-1)}\left(\sum_{i=1}^{\ell-1}\varphi(s_i)+\varphi(x_\ell) + \sum_{j=1}^{d-\ell}\varphi(y_j)\right) \\ 
			&\qquad - \psi^{(\ell-1)}\left( \sum_{i=1}^{\ell-1}\varphi(s_i)+\varphi(\delta) + \sum_{j=1}^{d-\ell}\varphi(y_j)\right) \Biggr) \ \mathrm{d}\lambda_{\ell-1}(\mathbf{s}_{1:\ell-1}) \\
			&= \int_{(\mathbf{0},\mathbf{x}_{1:\ell-1}]} \prod_{i=1}^{\ell}\varphi'(s_i) \psi^{(\ell-1)}\left(\sum_{i=1}^{\ell-1}\varphi(s_i)+\varphi(x_\ell) + \sum_{j=1}^{d-\ell}\varphi(y_j)\right)\ \mathrm{d}\lambda_{\ell-1}(\mathbf{s}_{1:\ell-1}).
		\end{align*}		
		Iterating this procedure another $\ell-2$ times finally results in
		\begin{align*}
			(\star) &= \int_{(0,x_1]} \varphi'(s_1) \psi'\left(\varphi(s_1) + \sum_{i=2}^{\ell}\varphi(x_i) + \sum_{j=1}^{d-\ell}\varphi(y_j)\right) \\ &=\psi\left(\sum_{i=1}^{\ell}\varphi(x_i) + \sum_{j=1}^{d-\ell}\varphi(y_j)\right) = C_\psi(\mathbf{x},\mathbf{y})
		\end{align*}
		as required.
	\end{proof}

	\section{Convergence in $\mathcal{C}_\text{ar}^d$}
	\label{sec:convergence:D1}
	In this section we extend the results in \cite{wcc} characterizing uniform convergence in the space of two-dimensional Archimedean copulas to the multivariate setting. There already exist several characterizations in terms of the corresponding generators, inverse generators, generating probability measures, marginal copulas and marginal densities (\cite{convMultiArch}, also see \cite{convArch}). In what follows, we contribute the items of convergence w.r.t. the metric $D_1$ (and $D_2,D_\infty$) as well as $1$-weak conditional convergence to the list of characterizing properties. Thereby, given $d$-dimensional copulas $C, C_1, C_2,\ldots$ with corresponding $\ell$-Markov kernels $K_C, K_{C_1}, K_{C_2},\ldots$, where $\ell\in\{1,\ldots,d-1\}$, we say that $(C_n)_{n\in\mathbb{N}}$ converges $\ell$-\emph{weakly conditional} to $C$ if, and only if, there exists a set $\Lambda\in\mathcal{B}(\mathbb{I}^\ell)$ with $\mu_{C^{1:\ell}}(\Lambda) = 1$ such that for every $\mathbf{x}\in\Lambda$ the sequence $(K_{C_n}(\mathbf{x},\cdot))_{n\in\mathbb{N}}$ of probability measures on $\mathcal{B}(\mathbb{I}^{d-\ell})$ converges weakly to the probability measure $K_C(\mathbf{x},\cdot)$ (cf. \cite[Section 4]{qmd}). Considering $\ell>1$, however, this notion may fail to be reasonable in full generality as for arbitrary $A,B\in\mathcal{C}^d$ it is possible that $\mu_{A^{1:\ell}}$ and $\mu_{B^{1:\ell}}$ are singular with respect to each other, and being only defined uniquely $\mu_{A^{1:\ell}}$-almost everywhere, $\mu_{B^{1:\ell}}$-almost everywhere, respectively, comparing $K_A(\mathbf{x},\cdot)$ and $K_B(\mathbf{x},\cdot)$ is nonsensical. 
	As pointed out in \cite{convMultiArch}, however, for special classes of copulas such as the class of copulas with identical marginals or Archimedean copulas this notion does make sense. Since every Archimedean copula $C$ has absolutely continuous marginals and its $\ell$-marginal density $c^{1:\ell}$ fulfills $c^{1:\ell}(\mathbf{x})>0$ for every $\mathbf{x}\in\mathbb{I}^\ell\setminus L_0^{1:\ell}$, the marginal measures $\mu_{C^{1:\ell}_1}$, $\mu_{C^{1:\ell}_2}$ of two Archimedean copulas $C_1,C_2$ can not be singular with respect to each other.

	The following characterization of uniform convergence for multivariate Archimedean copulas is due to \cite{convMultiArch}:
	\begin{theorem}\label{th:multi:arch:paper}
		Suppose that $C,C_1,C_2,\ldots$ are $d$-dimensional Archimedean copulas with generators $\psi, \psi_1, \psi_2,\ldots$, respectively. Then the following assertions are equivalent:
		\begin{enumerate}
			\item $(C_n)_{n\in\mathbb{N}}$ converges uniformly to $C$,
			\item $(C_n^{i:j})_{n\in\mathbb{N}}$ converges uniformly to $C^{i:j}$ for any $1\leq i < j \leq d$,
			\item $(\varphi_n)_{n\in\mathbb{N}}$ converges pointwise to $\varphi$ on $(0,1]$,
			\item $(\psi_n)_{n\in\mathbb{N}}$ converges uniformly to $\psi$ on $[0,\infty)$,
			\item $(\psi_n^{(m)})_{n\in\mathbb{N}}$ converges pointwise to $\psi^{(m)}$ on $(0,\infty)$ for $m\in\{1,\ldots, d-2\}$ and $(D^-\psi_n^{(d-2)})_{n\in\mathbb{N}}$ converges pointwise to $D^-\psi^{(d-2)}$ on \emph{Cont}$(D^-\psi^{(d-2)})$,
			\item $(c_n^{1:m})_{n\in\mathbb{N}}$ converges to $c^{1:m}$ a.e. in $\mathbb{I}^{m}$, $m\in\{2,\ldots, d-1\}$.
		\end{enumerate}
	\end{theorem}
	
	As direct consequence of Theorem \ref{th:interkernel:arch} (a version of) the 1-Markov kernel of an Archimedean copula $C_\psi$ is given by
	\begin{align*}
		K_\psi(x_1, \bigtimes_{i=2}^d[0,x_i]) = \begin{cases}
			1 & \text{ if } x_1\in\{0,1\} \\
			0 & \text{ if } x_1\in(0,1), \mathbf{x}\in\text{int}L_0 \\
			\frac{\psi'\left( \sum_{i=1}^d \varphi(x_i) \right)}{\psi'\left( \varphi(x_1) \right)} & \text{ if } x_1\in(0,1), \mathbf{x}\not\in\text{int}L_0
		\end{cases}.
	\end{align*}
	
	\noindent The following lemma contributes a characterization of pointwise/uniform convergence of Archimedean copulas in terms of $1$-Markov kernels (i.e., $(d-1)$-dimensional conditional distributions):
	
	\begin{Lemma}\label{lem:1wcc:arch}
		Suppose that $C_\psi,C_1,C_2,\ldots$ are $d$-dimensional Archimedean copulas with generators $\psi, \psi_1, \psi_2,\ldots$ and 1-Markov kernels $K_\psi, K_1, K_2,\ldots$, respectively. If $(C_n)_{n\in\mathbb{N}}$ converges uniformly to $C$ then for every $x_1\in(0,1)$
		\begin{align*}
			\lim_{n\to\infty} K_n(x_1, \bigtimes_{i=2}^d[0,x_i]) = K_\psi(x_1, \bigtimes_{i=2}^d[0,x_i])
		\end{align*}
		holds for every $(x_2,\ldots,x_d)\in \mathbb{I}^{d-1}$.
	\end{Lemma}
	\begin{proof}
		Suppose $C$ is non-strict and fix $x_1\in(0,1)$. First, consider the $x_1$-section of $L_0$ and let $(x_2,\ldots,x_d)\in(0,1)^{d-1}\cap (\text{int}L_0$)$_{x_1}$, i.e., $\sum_{i=1}^d \varphi(x_i) > \varphi(0)$. Defining $z_n := \sum_{i=1}^d \varphi_n(x_i)$, $z := \sum_{i=1}^d\varphi(x_i)$ and applying Theorem \ref{th:multi:arch:paper}, the sequence $\left(z_n\right)_{n\in\mathbb{N}}$ converges to $z > \varphi(0)$ and according to \cite[Lemma 4.3, Lemma 4.4]{convMultiArch} we have $\psi'_n\left(z_n\right) \to \psi'\left(z\right) = 0$ for $n\to\infty$. Additionally, $\psi'_n(\varphi_n(x_1))\to\psi'(\varphi(x_1))$ as $n\to\infty$ and taking together we get $K_n(x_1, \bigtimes_{i=2}^d[0,x_i]) \to K_\psi(x_1, \bigtimes_{i=2}^d[0,x_i])$ for $n\to\infty$. Second, consider $(x_2,x_3,\ldots,x_d)\in(0,1)^{d-1}\setminus(\text{int}L_0$)$_{x_1}$. Again using $z_n$, $z$ defined as above  (in this case we have $z < \varphi(0)$) and \cite[Lemma 4.3]{convMultiArch} we directly obtain $\psi_n'(z_n)\to\psi'(z)$ as well as $\psi'_n(\varphi_n(x_1))\to\psi'(\varphi(x_1))$ as $n\to\infty$ from which convergence of the $1$-Markov kernels follows. For strict limit copula $C$ we have int$L_0=\varnothing$ and the assertion results in the same way as for the second case above.
	\end{proof}
	
	\noindent Furthermore, according to \cite[Lemma 3, Lemma 5]{p15} $1$-weak conditional convergence implies convergence w.r.t. $D_1$, $D_2$ and w.r.t. $D_\infty$ which in turn all imply uniform convergence. Incorporating these facts directly yields the following extension of Theorem \ref{th:multi:arch:paper}:

	\begin{theorem}\label{th:main:res:arch}
		Suppose that $C,C_1,C_2,\ldots$ are $d$-dimensional Archimedean copulas. Then any of the statements in Theorem \ref{th:multi:arch:paper} is equivalent to any of the following:
		\begin{enumerate}
			\item $(C_n)_{n\in\mathbb{N}}$ converges to $C$ w.r.t. $D_i$, $i = 1,2$,
			\item $(C_n)_{n\in\mathbb{N}}$ converges to $C$ w.r.t. $D_\infty$,
			\item $(C_n)_{n\in\mathbb{N}}$ converges $1$-weakly conditional to $C$.
		\end{enumerate}
	\end{theorem}
	
	Having all Markov kernels of Archimedean copulas at hand, the next corollary says that within the class of $d$-dimensional Archimedean copulas uniform convergence (and thus any of the equivalent conditions) implies $\ell$-weak conditional convergence for every $\ell\in\{1,\ldots,d-1\}$, constituting an additional way of extending \cite[Theorem 4.1]{convMultiArch}.
	
	\begin{Cor}\label{cor:implication:interkernel:conv}
		Suppose that $C,C_1,C_2,\ldots$ are $d$-dimensional Archimedean copulas with generators $\psi, \psi_1, \psi_2,\ldots$, respectively. If $(C_n)_{n\in\mathbb{N}}$ converges uniformly to $C$ then for any $\ell\in\{1,\ldots,d-1\}$ there exists a set $\Lambda\in\mathcal{B}(\mathbb{I}^\ell)$ with $\mu_{C^{1:\ell}}(\Lambda)=1$ such that for all $\mathbf{x}\in\Lambda$ we have
		\begin{align*}
			\lim_{n\to\infty}K_n(\mathbf{x}, [\mathbf{0},\mathbf{y}]) = K_\psi(\mathbf{x},[\mathbf{0},\mathbf{y}])
		\end{align*}
		for every $\mathbf{y}\in U^\mathbf{x}$, where $U^\mathbf{x}$ is dense in $\mathbb{I}^{d-\ell}$.
	\end{Cor}
	\begin{proof}
		The case $\ell = d-1$ can be found in \cite{convMultiArch} while Lemma \ref{lem:1wcc:arch} deals with the case $\ell = 1$. Suppose $\ell\in\{2,\ldots,d-2\}$ then considering that $\psi^{(\ell)}$ is continuous on $(0,\infty)$ and setting $\Lambda := \mathbb{I}^\ell\setminus$ $L_0^{1:\ell}$ yields $\mu_{C^{1:\ell}}(\Lambda) = 1$. Now fix $\mathbf{x}\in\Lambda$. Proceeding as in Lemma \ref{lem:1wcc:arch} and distinguishing the cases $\mathbf{y}\in (\text{int}L_0)_\mathbf{x}$ and $\mathbf{y}\in \mathbb{I}^{d-\ell}\setminus(\text{int}L_0)_\mathbf{x}$ the result follows.
	\end{proof}

	\section{Singularity in $\mathcal{C}_\text{ar}^d$}
	\label{sec:singularity}
	Turning to singularity aspects, the main goal of this short section is to show that a $d$-dimensional copula $C$ with absolutely continuous $\ell$-marginal is singular if, and only if, its $\ell$-Markov kernel is singular w.r.t. $\lambda_{d-\ell}$. As direct consequence we obtain the fact that an Archimedean copula is singular if, and only if, its $\ell$-Markov kernel is singular w.r.t. $\lambda_{d-\ell}$ for \emph{every} $\ell\in\{1,\ldots,d-1\}$. \\
	According to \cite{Lange} the $\ell$-Markov kernel $K_C(\mathbf{x},\cdot)$ of a $d$-dimensional copula $C$ can be decomposed into the sum of two substochastic kernels $K_C^\text{abs}(\mathbf{x},\cdot)$ and $K_C^\text{sing}(\mathbf{x},\cdot)$, i.e. for every $\mathbf{x}\in\mathbb{I}^\ell$ and $E\in\mathcal{B}(\mathbb{I}^{d-\ell})$ it holds that
	\begin{align*}
		K_C(\mathbf{x}, E) = K_C^\text{abs}(\mathbf{x}, E) + K_C^\text{sing}(\mathbf{x}, E),
	\end{align*}
	where the measure $K_C^\text{abs}(\mathbf{x}, \cdot)$ is absolutely continuous w.r.t. $\lambda_{d-\ell}$ and the measure $K_C^\text{sing}(\mathbf{x},\cdot)$ is singular w.r.t. $\lambda_{d-\ell}$. 
	
	The next theorem is this section's main result:
	
	\begin{theorem}\label{th:singular:kernels}
		Suppose that $C$ is a $d$-dimensional copula with absolutely continuous $\ell$-marginal for some $\ell\in\{2,\ldots,d-1\}$. Then $C$ is singular (w.r.t. to $\lambda_{d}$) if, and only if, there exists a set $\Lambda\in\mathcal{B}(\mathbb{I}^\ell)$ with $\mu_{C^{1:\ell}}(\Lambda) = 1$ such that for all $\mathbf{x}\in\Lambda$ we have $K_C(\mathbf{x}, \cdot)$ is singular w.r.t. $\lambda_{d-\ell}$.
	\end{theorem}
	\begin{proof}
		Assuming singularity of $C$ there exists a set $E\in\mathcal{B}(\mathbb{I}^d)$ such that $\mu_C$ concentrates its mass on $E$ while $E$ is a $\lambda_d$-null set. 
		Using disintegration twice yields, on the one hand, the existence of a set $\Upsilon_1$ with $\mu_{C^{1:\ell}}(\Upsilon_1) = 1$ such that for all $\mathbf{x}\in\Upsilon_1$ we have $K_C(\mathbf{x},E_\mathbf{x}) = 1$ and, on the other hand, a set $\Upsilon_2$ with $\lambda_\ell(\Upsilon_2) = 1$ such that for all $\mathbf{x}\in\Upsilon_2$ we have $\lambda_{d-\ell}(E_\mathbf{x}) = 0.$
		
		Further, $
			\mu_{C^{1:\ell}}(\Upsilon_2) = \int_{\Upsilon_2} c^{1:\ell} \ \mathrm{d}\lambda_\ell = \int_{\mathbb{I}^\ell} c^{1:\ell} \ \mathrm{d}\lambda_\ell = 1.
		$
		Considering that $\mu_{C^{1:\ell}}(\Upsilon_1\cap \Upsilon_2) = 1$ the desired property follows.
		
		Assuming the converse, for every $\mathbf{x}\in\Lambda$ we have $K_C(\mathbf{x},\cdot)\perp\lambda_{d-\ell}$ and hence also $K_C^{\text{abs}}(\mathbf{x},\mathbb{I}^{d-\ell}) = 0$. We show that the absolutely continuous part $\mu_C^\text{abs}$ of the Lebesgue decomposition of $\mu_C$ vanishes. Again denoting by $c^{1:d}$ the Radon-Nikodym derivative of $\mu_C$ w.r.t. $\lambda_d$, Lemma \ref{lem:abs:cts:kernel} in \ref{sec:aux:copulas} directly yields
		\begin{align*}
			0 &= \int_\Lambda K_C^\text{abs}(\mathbf{x},\mathbb{I}^{d-\ell}) \ \mathrm{d}\mu_{C^{1:\ell}}(\mathbf{x}) = \int_\Lambda \int_{\mathbb{I}^{d-\ell}} \frac{c^{1:d}(\mathbf{x},\mathbf{y})}{c^{1:\ell}(\mathbf{x})} \ \mathrm{d}\lambda_{d-\ell}(\mathbf{y}) \mathrm{d}\mu_{C^{1:\ell}}(\mathbf{x}) \\
			&= \int_{\mathbb{I}^\ell} \int_{\mathbb{I}^{d-\ell}} \frac{c^{1:d}(\mathbf{x},\mathbf{y})}{c^{1:\ell}(\mathbf{x})} \cdot c^{1:\ell}(\mathbf{x}) \ \mathrm{d}\lambda_{d-\ell}(\mathbf{y}) \mathrm{d}\lambda_\ell(\mathbf{x}) = \int_{\mathbb{I}^d} c^{1:d}(\mathbf{x},\mathbf{y}) \ \mathrm{d}\lambda_d(\mathbf{x},\mathbf{y}) = \mu_C^\text{abs}(\mathbb{I}^d)
		\end{align*}
		completing the proof.
	\end{proof}
	The following corollary is immediate:
	\begin{Cor}\label{cor:arch:singular}
		A $d$-dimensional Archimedean copula $C$ is singular if, and only if, every uni- and multivariate Markov kernel is singular w.r.t. the appropriate Lebesgue measure.
	\end{Cor}
	
	\begin{Rem}\label{rem:singular}
		The result in Corollary \ref{cor:arch:singular} is in strong contrast to the general case: Consider, for instance, the simple example of the $3$-dimensional copula $B(x,y,z) = z\cdot\min(x,y)$ with marginals $B^{12} = M, B^{13}=B^{23}=\Pi_2$ as in Figure \ref{fig:ex:singular}. Then for $\lambda_1$-a.e. $x\in\mathbb{I}$ the probability measure $K_B(x, \cdot)$ on $\mathcal{B}(\mathbb{I}^2)$ is singular w.r.t. $\lambda_2$ while for $\mu_M$-a.e. $(x,y)\in\mathbb{I}^2$ the probability measure $K_B(x,y,\cdot)$ on $\mathcal{B}(\mathbb{I})$ is absolutely continuous w.r.t. $\lambda_1$.
		\begin{figure}[htp!]
			\centering 	
			\includegraphics[width=0.35\textwidth]{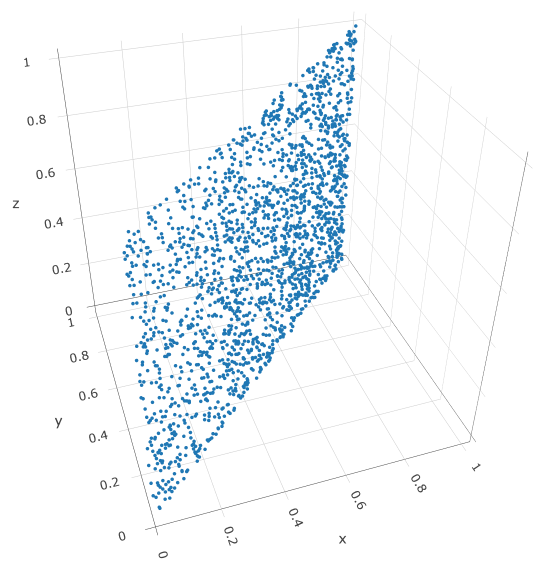}
			\caption{Sample of size $n=2000$ of the $3$-dimensional copula $B$ as considered in Remark \ref{rem:singular}.}
			\label{fig:ex:singular}
		\end{figure}
	\end{Rem}

	\section{Conditional Archimedean copulas from a Markov kernel perspective}\label{sec:conditional:copulas}
	As the dependence structure of a random vector may be heavily impacted by a (vector of) covariate(s) the introduction of conditional copulas is natural to determine the variation in dependence when changing the values of the covariate(s). 
	%	Therefore, conditional Archimedean copulas are covered in several publications, see, e.g., \cite{conditionalCop_Sungur, ConditionalCopula, juri_taildepdenceCopula, gijbels_conditionalCop}. 
	We particularly follow \cite{ConditionalCopula} where the authors study various properties of conditional Archimedean copulas. In contrast, however, we will introduce these conditional dependence structures largely from a Markov kernel perspective which leads to several beautiful relations, and present an alternative proof of the truncation invariance property while working with general strict $d$-monotone generators instead of requiring complete monotonicity. Furthermore, the results of the previous sections are used to show that convergence behaviour, singularity as well as conditional increasingness carry over from the original Archimedean copulas to their conditional copulas. Additionally, we demonstrate that in this scenario estimating (the generator of) the original copula suffices to get an estimator of its conditional copula. From there, we propose a  \textquoteleft conditional dependence' measure (in analogy to conditional Kendall's tau or conditional Spearman's rho, see, e.g., \cite{gijbels_partial_copulas, gijbels_conditionalCop}) as an alternative way to obtain additional information of the dependence in data from an Archimedean model given covariate values.

	We start with the following simple lemma where we conduct a slight abuse of notation and simply set 
	\begin{align*}
		K_C(\mathbf{x},[0,y]) = K_C(\mathbf{x},\mathbb{I}\times\ldots\times\mathbb{I}\times [0,y]\times\mathbb{I}\ldots\times\mathbb{I})
	\end{align*}
	for $\mathbf{x}\in\mathbb{I}^\ell,y\in\mathbb{I}$.
	
	\begin{Lemma}\label{lem:connection:kernels}
		Suppose that $C$ is a $d$-dimensional Archimedean copulas with strict generator $\psi$ and let $\ell\in\{1,\ldots,d-1\}$. Then
		\begin{align*}
			K_C(\mathbf{x}, [0,y_{\ell+1}]\times\ldots\times[0,y_d]) &= K_C(\mathbf{x},[0,C^{\ell+1:\ell+2}(y_{\ell+1},y_{\ell+2})] \times\ldots\times [0,C^{d-1:d}(y_{d-1}, y_d)]) \\ 
			&= \ldots \\
			&= K_C(\mathbf{x},[0,y_{\ell+1}]\times[0,C^{\ell+2:d}(y_{\ell+2}, \ldots, y_d)]) \\
			&= \ldots \\
			&= K_C(\mathbf{x},[0,C^{\ell+1:d}(y_{\ell+1}, \ldots, y_d)]).
		\end{align*}
		Therefore, consolidating the rectangles via the appropriate marginal copulas we can vary the dimension of the measured set of the Markov kernel.
	\end{Lemma}
	\begin{proof}
		Since $\psi$ is strict we have $\varphi(\psi(z)) = z$ for all $z\in[0,\infty)$ so is is straightforward to verify 
		%		all identities follow immediately from the fact that all expressions evaluate to
		\begin{align*}
			K_C(\mathbf{x}, [0,y_{\ell+1}]\times\ldots\times[0,y_d]) &= \frac{\psi^{(\ell)}\left( \sum_{i=1}^\ell\varphi(x_i) + \sum_{j=\ell+1}^d \varphi(y_j) \right)}{\psi^{(\ell)}\left( \sum_{i=1}^\ell\varphi(x_i) \right)} \\
			&= \frac{\psi^{(\ell)}\left( \sum_{i=1}^\ell\varphi(x_i) + \varphi\left(C^{\ell+1:d}(y_{\ell+1},\ldots,y_d)\right) \right)}{\psi^{(\ell)}\left( \sum_{i=1}^\ell\varphi(x_i) \right)} \\
			&= K_C(\mathbf{x},[0,C^{\ell+1:d}(y_{\ell+1}, \ldots, y_d)])
		\end{align*}
		and similarly for all other expressions. Note that for $\ell=d-1$ the same holds true by using the left-hand derivative $D^-\psi^{(d-2)}$ of the $(d-2)$-nd derivative of $\psi$.
	\end{proof}
	
	From now on we consider a strict $d$-dimensional Archimedean copula $C$ and let $\ell\in\{1,\ldots,d-2\}$ (notice that therefore $\psi^{(\ell)}$ is continuous and invertible). The $\ell$-Markov kernel of $C$ (as given in eq. \eqref{eq:interkernel:arch}) being a $(d-\ell)$-dimensional distribution function for any fixed $\mathbf{x}\in(0,1)^\ell$, Sklar's Theorem (see, e.g., \cite{Principles}) yields the existence of some copula $C^\mathbf{x}\in\mathcal{C}^{d-\ell}$ such that both,
	\begin{align*}
		&K_C(\mathbf{x},[\mathbf{0},\mathbf{y}]) = C^\mathbf{x}(K_C(\mathbf{x},[0,y_1]), \ldots, K_C(\mathbf{x},[0,y_{d-\ell}])), \\
		&C^\mathbf{x}(u_1,\ldots, u_{d-\ell}) = K_C(\mathbf{x}, [0,g_\mathbf{x}^{-1}(u_1)]\times\ldots\times[0,g_\mathbf{x}^{-1}(u_{d-\ell})]),
	\end{align*}
	where, again slightly abusing notation, $g_\mathbf{x}(y) = K_C(\mathbf{x},[0,y])$, $y\in[0,1]$, is the marginal distribution of $K_C(\mathbf{x},\cdot)$ and $g_\mathbf{x}^{-1}$ its inverse. Knowing the form of all Markov kernels, we compute  $g_\mathbf{x}^{-1}$ as
	\begin{align*}
		g_\mathbf{x}^{-1}(u) = \psi\left( \psi^{(\ell)^{-1}}\left( u\cdot \psi^{(\ell)}\left( \sum_{i=1}^\ell \varphi(x_i) \right) \right) - \sum_{i=1}^\ell\varphi(x_i) \right)
	\end{align*}
	for $u\in\mathbb{I}$ and 
	\begin{align*}
		C^\mathbf{x}(\mathbf{u}) = \frac{\psi^{(\ell)}\left( \sum_{i=1}^\ell\varphi(x_i) + \sum_{j=1}^{d-\ell}\varphi(g_\mathbf{x}^{-1}(u_j)) \right)}{\psi^{(\ell)}\left( \sum_{i=1}^\ell\varphi(x_i) \right)}
	\end{align*}
	for $\mathbf{u}\in\mathbb{I}^{d-\ell}$. Notice that according to Lemma \ref{lem:connection:kernels} we can additionally write
	\begin{align*}
		C^\mathbf{x}(\mathbf{u}) = K_C(\mathbf{x}, [0,C^{1:d-\ell}(g_\mathbf{x}^{-1}(u_1),\ldots, g_\mathbf{x}^{-1}(u_{d-\ell}))])
	\end{align*}
	which will turn out to be useful in the sequel.
	
	In \cite{ConditionalCopula} the surprising fact is stated that $C^\mathbf{x}$ is again of Archimedean type if $C$ is induced by a completely monotone generator (the so-called truncation invariance property). Not only will we get rid of the restriction to completely monotone generators but offer an alternative approach to deriving this assertion. We pick up the idea in \cite{juri_taildepdenceCopula} where the authors used the well-known result \cite[Theorem 4.1.6]{Nelsen}, saying that a \emph{bivariate} associative copula $B$ fulfilling $B(x,x)<x$ for all $x\in(0,1)$ is necessarily Archimedean, to prove that the so-called lower tail dependence copula of Archimedean copulas is again Archimedean. Similar conditions hold in the general case. In fact, according to \cite{multiAsso} a $d$-dimensional copula $C$ that is associative and satisfies $C(x,\ldots,x) < x$ for every $x\in(0,1)$ is necessarily Archimedean. Thereby, a $d$-variate function $f:\mathbb{I}^d\to\mathbb{I}$ is defined to be associative if, and only if, for all $x_1,x_2,\ldots, x_{2d-1}\in\mathbb{I}$ we have
	\begin{align*}
		f(f(x_1,\ldots,x_d),x_{d+1},\ldots, x_{2d-1}) &= f(x_1,f(x_2,\ldots,x_{d+1}),x_{d+2},\ldots, x_{2d-1}) \\
		&= \ldots \\
		&= f(x_1,\ldots,x_{d-1},f(x_d,\ldots,x_{2d-1})).
	\end{align*}
	We will now use this characterization to obtain a nice and elementary derivation of the fact that the conditional copula of an Archimedean copula is again Archimedean:
	\begin{Prop}
		Suppose $C$ is a strict Archimedean copula and $\ell\in\{1,\ldots,d-2\}$. For any $\mathbf{x}\in(0,1)^\ell$ the conditional copula $C^\mathbf{x}$ of $C$ is a $(d-\ell)$-dimensional Archimedean copula.
	\end{Prop}
	\begin{proof}
		We start with associativity and suppose $u_1,\ldots,u_{2(d-\ell)-1}\in\mathbb{I}$. Plugging in and cancelling gives 
		\begin{align*}
			\varphi(g_\mathbf{x}^{-1}&(C^\mathbf{x}(u_1,\ldots,u_{d-\ell}))) \\
			&=  \varphi\left(\psi\left( \psi^{(\ell)^{-1}}\left( \tfrac{\psi^{(\ell)}\left( \sum_{i=1}^\ell\varphi(x_i) + \sum_{j=1}^{d-\ell}\varphi(g_\mathbf{x}^{-1}(u_j)) \right)}{\psi^{(\ell)}\left( \sum_{i=1}^\ell\varphi(x_i) \right)} \cdot \psi^{(\ell)}\left( \sum_{i=1}^\ell \varphi(x_i) \right) \right)  - \sum_{i=1}^\ell\varphi(x_i) \right) \right) \\
			&= \sum_{j=1}^{d-\ell}\varphi(g_\mathbf{x}^{-1}(u_j)).
		\end{align*}
		Thus, it follows that 
		\begin{align*}
			C^\mathbf{x}(C^\mathbf{x}(u_1,\ldots,u_{d-\ell}),& u_{d-\ell+1},\ldots,u_{2(d-\ell)-1}) \\
			&= \tfrac{\psi^{(\ell)}\left( \sum_{i=1}^\ell\varphi(x_i) + \sum_{j=1}^{d-\ell}\varphi(g_\mathbf{x}^{-1}(u_j)) + \sum_{r=d-\ell+1}^{2(d-\ell)-1}\varphi(g_\mathbf{x}^{-1}(u_r)) \right)}{\varphi^{(\ell)}\left(\sum_{i=1}^\ell \varphi(x_i)\right)} \\
			&= \tfrac{\psi^{(\ell)}\left( \sum_{i=1}^\ell\varphi(x_i) + \sum_{j=1}^{2(d-\ell)-1}\varphi(g_\mathbf{x}^{-1}(u_j)) \right)}{\varphi^{(\ell)}\left(\sum_{i=1}^\ell \varphi(x_i)\right)} \\
			&= C^\mathbf{x}(u_1,C^\mathbf{x}(u_2,\ldots,u_{d-\ell+1}), u_{d-\ell+2},\ldots, u_{2(d-\ell)-1}) \\
			&= \ldots \\
			&= C^\mathbf{x}(u_1,\ldots,u_{d-\ell-1}, C^\mathbf{x}(u_{d-\ell}, \ldots, u_{2(d-\ell)-1}))
		\end{align*}
		and $C^\mathbf{x}$ is associative. For the diagonal property let $u\in(0,1)$ and assume $\ell$ is even, so that $\psi^{(\ell)}\geq 0$ and decreasing (the odd case follows analogously). Then it is straightforward to verify that $C^\mathbf{x}(u,\ldots, u) < u$ holds if, and only if,
		\begin{align*}
			\psi^{(\ell)}&\left( \sum_{i=1}^\ell\varphi(x_i) + (d-\ell)\psi^{{(\ell)}^{-1}}\left( u \cdot \psi^{(\ell)}\left(\sum_{i=1}^\ell\varphi(x_i)\right) \right) - (d-\ell)\sum_{i=1}^\ell\varphi(x_i) \right) \\
			&< u\cdot \psi^{(\ell)}\left(\sum_{i=1}^\ell\varphi(x_i)\right)
		\end{align*}
		which is the case if, and only if, $u < 1$ as desired. Now, as direct consequence of \cite[Theorem 4.2]{multiAsso} we have $C^\mathbf{x}\in\mathcal{C}^{d-\ell}_\text{ar}$.
	\end{proof}
	
	In contrast to \cite{ConditionalCopula}, we use a multivariate analogue of Nelsen's method \cite[Theorem 4.3.8]{Nelsen} to determine the (inverse) generator of $C^\mathbf{x}$ and give a representation in terms of the Markov kernel which will lead to nice interrelations between (conditional) Archimedean generators:
	
	\begin{Lemma}
		Suppose $C$ is a strict $d$-dimensional Archimedean copula with generator $\psi$ and let $\ell\in\{1,\ldots,d-2\}$. For any $\mathbf{x}\in(0,1)^\ell$ the Archimedean generator of the conditional copula $C^\mathbf{x}$ of $C$ is given by
		\begin{align*}
			\psi^\mathbf{x}(z) = K_C(\mathbf{x},[0,\psi(z)])
		\end{align*}
		for every $z\in[0,\infty)$.
	\end{Lemma}
	\begin{proof}
		Similarly to \cite[Theorem 4.3.8]{Nelsen} it holds that for $\mathbf{u}\in\mathbb{I}^{d-\ell}$ we have
		\begin{align*}
			\frac{(\varphi^\mathbf{x}(u_1))'}{(\varphi^\mathbf{x}(u_2))'} 
			%		&= \frac{\psi^{(\ell+1)}\left( \sum_{i=1}^\ell\varphi(x_i) + \sum_{j=1}^{d-\ell}\varphi(g_\mathbf{x}^{-1}(u_j)) \right) \cdot \varphi'(g_\mathbf{x}^{-1}(u_1))\cdot (g_\mathbf{x}^{-1})'(u_1)}{\psi^{(\ell+1)}\left( \sum_{i=1}^\ell\varphi(x_i) + \sum_{j=1}^{d-\ell}\varphi(g_\mathbf{x}^{-1}(u_j)) \right) \cdot \varphi'(g_\mathbf{x}^{-1}(u_2))\cdot (g_\mathbf{x}^{-1})'(u_2)} \\
			= \frac{\varphi'(g_\mathbf{x}^{-1}(u_1))\cdot (g_\mathbf{x}^{-1})'(u_1)}{ \varphi'(g_\mathbf{x}^{-1}(u_2))\cdot (g_\mathbf{x}^{-1})'(u_2)}
		\end{align*}
		for $\lambda$-almost all $u_1,u_2\in(0,1)$. Hence, $(\varphi^{\mathbf{x}})'(s) = \alpha_\mathbf{x} \varphi'(g_\mathbf{x}^{-1}(s)) (g_\mathbf{x}^{-1})'(s)$ where $\alpha_\mathbf{x}$ is a constant only depending on $\mathbf{x}$. Integrating and substituting $t = g_\mathbf{x}^{-1}(s)$ yields
		\begin{align*}
			\varphi^\mathbf{x}(s) = -\int_{[s,1]}(\varphi^\mathbf{x})'(t) \ \mathrm{d}\lambda(t) = \alpha_\mathbf{x}\int_{[g_\mathbf{x}^{-1}(s), 1]} \varphi'(t) \ \mathrm{d}\lambda(t) = \alpha_\mathbf{x} \varphi(g_\mathbf{x}^{-1}(s)).
		\end{align*}
		As Archimedean generators are only unique up to a multiplicative constant (see, e.g., \cite{Nelsen}) we set $\alpha_\mathbf{x} = 1$ and arrive at 
		\begin{align*}
			\psi^\mathbf{x}(z) = K_C(\mathbf{x},[0,\psi(z)]).
		\end{align*} 
		It follows from the properties of $K_C(\mathbf{x},\cdot)$ and $\psi^{(\ell)}$ that $\psi^\mathbf{x}$ is strictly decreasing and continuous on $[0,\infty)$, fulfills $\psi^\mathbf{x}(0) = 1$ as well as $\lim_{z\to\infty}\psi^\mathbf{x}(z) = 0$. Thus, for $\mathbf{u}\in\mathbb{I}^{d-\ell}$ Lemma \ref{lem:connection:kernels} gives
		\begin{align*}
			\psi^\mathbf{x}\left(\sum_{i=1}^{d-\ell}\varphi^\mathbf{x}(u_i)\right) &= K_C\left(\mathbf{x}, \left[0, \psi\left( \sum_{j=1}^{d-\ell}\varphi(g_\mathbf{x}^{-1}(u_j)) \right)\right]\right) \\
			&= K_C\left( \mathbf{x}, \left[ 0, C^{1:d-\ell}(g_\mathbf{x}^{-1}(u_1), \ldots, g_\mathbf{x}^{-1}(u_{d-\ell})) \right] \right) \\
			&= K_C(\mathbf{x}, [0,g_\mathbf{x}^{-1}(u_1)]\times\ldots\times[0,g_\mathbf{x}^{-1}(z_{d-\ell})]) = C^\mathbf{x}(\mathbf{u}).
		\end{align*}
		Observe that $(d-\ell)$-monotonicity of $\psi^\mathbf{x}$ is also inherited from $\psi$: In fact, according to \cite[Proposition 2.3]{multiArchNeslehova} it suffices to check that $(-1)^{d-\ell-2}(\psi^\mathbf{x})^{(d-\ell-2)}$ exists on $(0,\infty)$, is non-negative, decreasing and convex. Writing
		\begin{align*}
			(-1)^{d-\ell-2}(\psi^\mathbf{x})^{(d-\ell-2)}(z) &= (-1)^{d-\ell-2}\frac{\psi^{(d-2)}\left( \sum_{i=1}^\ell\varphi(x_i) + z \right)}{\psi^{(\ell)}\left( \sum_{i=1}^\ell\varphi(x_i) \right)} \\
			&= \frac{(-1)^{d-2}\psi^{(d-2)}\left( \sum_{i=1}^\ell\varphi(x_i) + z \right)}{(-1)^{\ell}\psi^{(\ell)}\left( \sum_{i=1}^\ell\varphi(x_i) \right)}
		\end{align*}
		all required properties follow directly.
	\end{proof}
	
	\begin{exa}\label{ex:cond:gen:sample}
		We consider the $3$-dimensional Gumbel copula with parameter $\vartheta$ which is generated by $\tilde{\psi}(z) = \exp(-z^{1/\vartheta})$. Thus, a normalized generator is given by $\psi_\vartheta(z) = \exp\left(-(z\cdot (-\log(1/2))^\vartheta)^{1/\vartheta}\right) = 2^{-z^{1/\vartheta}}$. Fixing $\vartheta = 3$ we calculated $\psi^x_5$ for several $x\in(0,1)$ and illustrated the results in Figure \ref{fig:ex:cond:gen:sample}. 
		\begin{figure}[htp!]
			\centering 	
			\includegraphics[width=\textwidth]{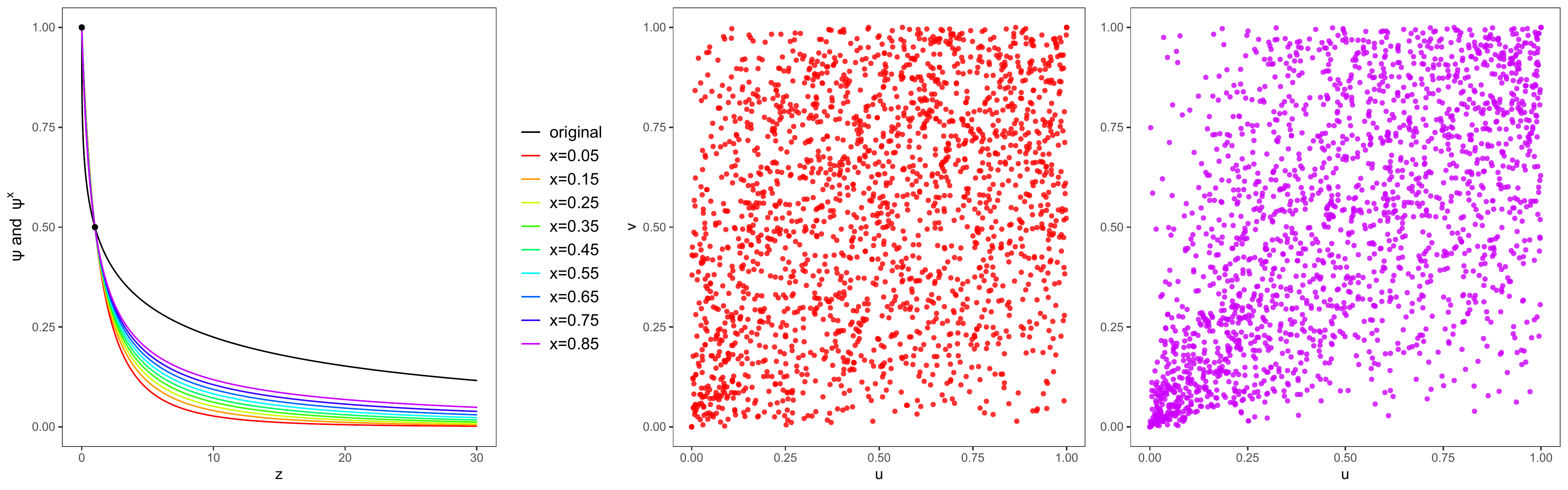}
			\caption{Archimedean generators of the original Gumbel copula $C_3$ and conditional generators $\psi_3^x$ for $x\in\{0.05,0.15,\ldots,0.85\}$ (left panel) and samples of size $n=2000$ of the conditional copulas $C_3^{0.05}$ and $C_3^{0.85}$ (middle and right panel) as considered in Example \ref{ex:cond:gen:sample}.}
			\label{fig:ex:cond:gen:sample}
		\end{figure}
		The left panel shows the original as well as the normalized conditional generators while the middle and right panels depict samples of $C^{0.05}_3$ and $C^{0.95}_3$, respectively.  
	\end{exa}

	\begin{Rem}
		Having the Markov kernel representation, the following connection between generator and conditional generator now is an immediate consequence of disintegration (cf. \cite[Section 4]{ConditionalCopula} for the $3$-dimensional case):
		\begin{align}\label{eq:connection:psi:psix}
			\psi(z) = \int_{\mathbb{I}^\ell} \psi^\mathbf{x}(z) \ \mathrm{d}\mu_{C^{1:\ell}}(\mathbf{x}).
		\end{align}
		The authors in \cite{ConditionalCopula} state that the conditional copula w.r.t.  any conditioning event $\mathbf{X}\in A\subseteq\mathbb{I}^\ell$ of an Archimedean copula is again Archimedean and, in particular, show that the corresponding inverse generator $\varphi^{[\mathbf{0},\mathbf{x}]}$ of the conditional copula $C^{[\mathbf{0},\mathbf{x}]}$ w.r.t. the conditioning $\mathbf{X} \leq \mathbf{x}$ is given by
		\begin{align*}
			\varphi^{[\mathbf{0},\mathbf{x}]}(z) = \varphi\left( z\cdot\psi\left( \sum_{i=1}^\ell \varphi(x_i) \right)  \right) - \varphi\left(\psi\left(\sum_{i=1}^\ell\varphi(x_i)\right)\right).
		\end{align*}
		Building upon that and rewriting $\varphi^{[\mathbf{0},\mathbf{x}]}$ we get $\varphi^{[\mathbf{0},\mathbf{x}]}(z) = \varphi\left( z\cdot C^{1:\ell}(\mathbf{x}) \right) - \varphi(C^{1:\ell}(\mathbf{x}))$ and hence
		\begin{align*}
			\psi^{[\mathbf{0},\mathbf{x}]} &= \frac{\psi(z+\varphi(C^{1:\ell}(\mathbf{x})))}{C^{1:\ell}(\mathbf{x})} = \frac{C^{1:\ell+1}(\mathbf{x}, \psi(z))}{C^{1:\ell}(\mathbf{x})}.		
		\end{align*}
		In analogy to eq. \eqref{eq:connection:psi:psix} disintegration now also yields the nice relation
		\begin{align*}
			\psi^{[\mathbf{0},\mathbf{x}]}(z) = \frac{1}{C^{1:\ell}(\mathbf{x})} \int_{[\mathbf{0},\mathbf{x}]} \psi^\mathbf{s}(z) \ \mathrm{d}\mu_{C^{1:\ell}}(\mathbf{s}).
		\end{align*}
	\end{Rem}
	
	\vspace*{0.3cm}
	
	It is well-known that imposing some regularity assumption (regular variation at infinity) on the generator $\psi$ of the Archimedean copula $C$, the conditional generator $\psi^\mathbf{x}$ converges to the Clayton copula with parameter $\rho/ (1+\rho\cdot \ell)$ as $\mathbf{x}\to\mathbf{0}$. From a different angle, given a sequence $(C_n)_{n\in\mathbb{N}}$ converging uniformly to $C$ within $\mathcal{C}^d_\text{ar}$ we now have the tools to answer the natural question whether the corresponding conditional copulas converge in that sense, too.
	
	\begin{theorem}\label{th:conv:conditional:cop}
		Suppose that $C_\psi,C_1,C_2,\ldots$ are $d$-dimensional Archimedean copulas with strict, normalized generators $\psi,\psi_1,\psi_2,\ldots$ and Markov kernels $K_\psi, K_1, K_2,\ldots$, respectively. If $(C_n)_{n\in\mathbb{N}}$ converges pointwise to $C$ then for $\ell\in\{1,\ldots,d-2\}$ and any fixed $\mathbf{x}\in(0,1)^\ell$ we have that for every $\mathbf{u}\in\mathbb{I}^{d-\ell}$ it holds
		\begin{align*}
			\lim_{n\to\infty} C^\mathbf{x}_n(\mathbf{u}) = C^\mathbf{x}_\psi(\mathbf{u}).
		\end{align*}
		
	\end{theorem}
	\begin{proof}
		Recall that according to Corollary \ref{cor:implication:interkernel:conv} we already have
		\begin{align*}
			\lim_{n\to\infty} K_n(\mathbf{x},[0,y]) = K_\psi(\mathbf{x},[0,y])
		\end{align*}
		for every $y\in\mathbb{I}$. 
		Considering that $\psi^{(\ell)}$ is continuous on $(0,\infty)$ and applying \cite[Lemma 4.4]{convMultiArch} it directly follows that
		\begin{align*}
			\lim_{n\to\infty}\psi_n^\mathbf{x}(z) = \lim_{n\to\infty} K_n(\mathbf{x},[0,\psi_n(z)]) = K_\psi(\mathbf{x},[0,\psi(z)]) = \psi^\mathbf{x}(z)
		\end{align*}
		for every $z\in[0,\infty)$. As continuous monotone functions the convergence is even uniform (see \cite{multiArchMassKernel}) and applying Theorem \ref{th:main:res:arch} completes the proof.
	\end{proof}
	
	Triggered by \cite{gijbels_partial_copulas} where the authors study a conditional version of Kendall's tau, we propose a conditional version of the multivariate, Markov kernel based dependence measure $\zeta_1 = \zeta_1^d$ introduced in \cite{qmd}. For $(\mathbf{X},Y)\sim C\in\mathcal{C}^d$ this dependence measure is defined by
	\begin{align*}
		\zeta_1^d(C) := 3\int_\mathbb{I}\int_{\mathbb{I}^{d-1}} \left| K_C(\mathbf{s},[0,y]) - y \right| \ \mathrm{d}\mu_{C^{1:d-1}}(\mathbf{s})\mathrm{d}\lambda(y).
	\end{align*}
	Considering a random vector $(\mathbf{X},\mathbf{Y})$, where the subvector $\mathbf{X}$ of covariates is $\ell$-dimensional, and setting 
	\begin{align*}
		\zeta_1^\mathbf{x}(C) = \zeta_1^{d-\ell}(C^\mathbf{x})
	\end{align*}
	yields a conditional multivariate dependence measure quantifying the degree of dependence of the conditional copula $C^\mathbf{x}$ of $C$ when fixing the values of the covariate $\mathbf{X}$. 
	
	Together with Lemma \ref{lem:conv:cond:zeta} in \ref{sec:aux:copulas} we now have the tools to show the remarkable fact that for fixed $\mathbf{x}\in(0,1)^\ell$ uniform convergence of a sequence $(C_n)_{n\in\mathbb{N}}$ of Archimedean copulas converging to $C\in\mathcal{C}^d_\text{ar}$ already implies convergence of the conditional dependence measures $(\zeta_1^\mathbf{x}(C_n))_{n\in\mathbb{N}}$ to $\zeta_1^\mathbf{x}(C)$. 
	
	\begin{Cor}\label{cor:conv:zeta1}
		Suppose that $C,C_1,C_2,\ldots$ are $d$-dimensional Archimedean copulas. If $(C_n)_{n\in\mathbb{N}}$ converges uniformly to $C$ then for every fixed $\mathbf{x}\in(0,1)^\ell$ the sequence of conditional copulas $(C_n^\mathbf{x})_{n\in\mathbb{N}}$ converges $(d-\ell-1)$-weakly conditional to $C^\mathbf{x}$. Furthermore we have
		\begin{align}\label{eq:conditional:zeta:conv}
			\lim_{n\to\infty}|\zeta_1^\mathbf{x}(C_n) - \zeta_1^\mathbf{x}(C)| = 0.
		\end{align}
	\end{Cor}
	\begin{proof}
		Combining Theorems \ref{th:main:res:arch}, \ref{th:conv:conditional:cop} as well as Lemma \ref{lem:conv:cond:zeta} yields both, $(d-\ell-1)$-weak conditional convergence of the conditional copulas and convergence of the conditional dependence measures.
	\end{proof}
	
	The practical relevance of our study is twofold and entailed by Theorem \ref{th:conv:conditional:cop} and Corollary \ref{cor:conv:zeta1}. First, in strong contrast to the estimation of general conditional copulas (see, e.g., \cite{gijbels_estimationConditionalCop, gijbels_conditionalCop}) in case of an Archimedean copula $C$ it suffices to estimate (the generator of) $C$ in order to obtain an estimator of (the generator of) its conditional copula $C^\mathbf{x}$. Additionally, an estimator of (the generator of) $C$ also gives rise to an estimation of the conditional dependence measure $\zeta_1^\mathbf{x}(C)$ which might be a powerful alternative to conditional association measures (as studied, e.g., in \cite{gijbels_partial_copulas,gijbels_conditionalCop}). \\
	Regarding the estimation of multivariate Archimedean copulas we refer to \cite{estimationArch_hering, estimationArch_knownMarginals}, and the references therein, for parametric methodologies and to \cite{GNZ} for a non-parametric procedure. It is worth remarking, however, that the approach from \cite{GNZ} is only proven to be valid in dimensions $2,3$ and $4$ whereas, despite strong evidence, the general case remains an open problem.
	
	The following example provides a small simulation in the $3$-dimensional case illustrating the estimation of $\zeta_1^x$ from data of an Archimedean parametric family for several choices of conditioning and different sample sizes.
	\begin{exa}\label{ex:est:zeta1x}
		Consider the $3$-dimensional Gumbel copula $C_\vartheta$ with generator $\psi(z) = \exp(-z^{1/\vartheta})$ and parameter $\vartheta = 5$. Working with the R-package \textquoteleft copula' (\cite{copulaPackage}) we generate samples of $C_5$ of size $n\in\{50,100,500,1000\}$ based on which we estimate its parameter to obtain an estimated generator $\hat{\psi}_n$ of $\psi$. Fixing covariate values $x\in\{0.05, 0.15, \ldots,0.95\}$ we directly compute $\hat{\psi}^x_n$ and consequently $\zeta_1^x(\hat{C}_n) = \zeta_1(\hat{C}_n^x)$. We repeated this process $R=300$ times and compared the real values of $\zeta_1^x(C_5)$ to the average values of $\zeta_1^x(\hat{C}_n)$ over all runs. The results are plotted in Figure \ref{fig:ex:est:zeta1x}. One can observe that for increasing $x$ the dependence increases and that the estimation works very well, with small deviations for higher values of $x$ which get better for higher sample sizes. 
		\begin{figure}[htp!]
			\centering 	
			\includegraphics[width=\textwidth]{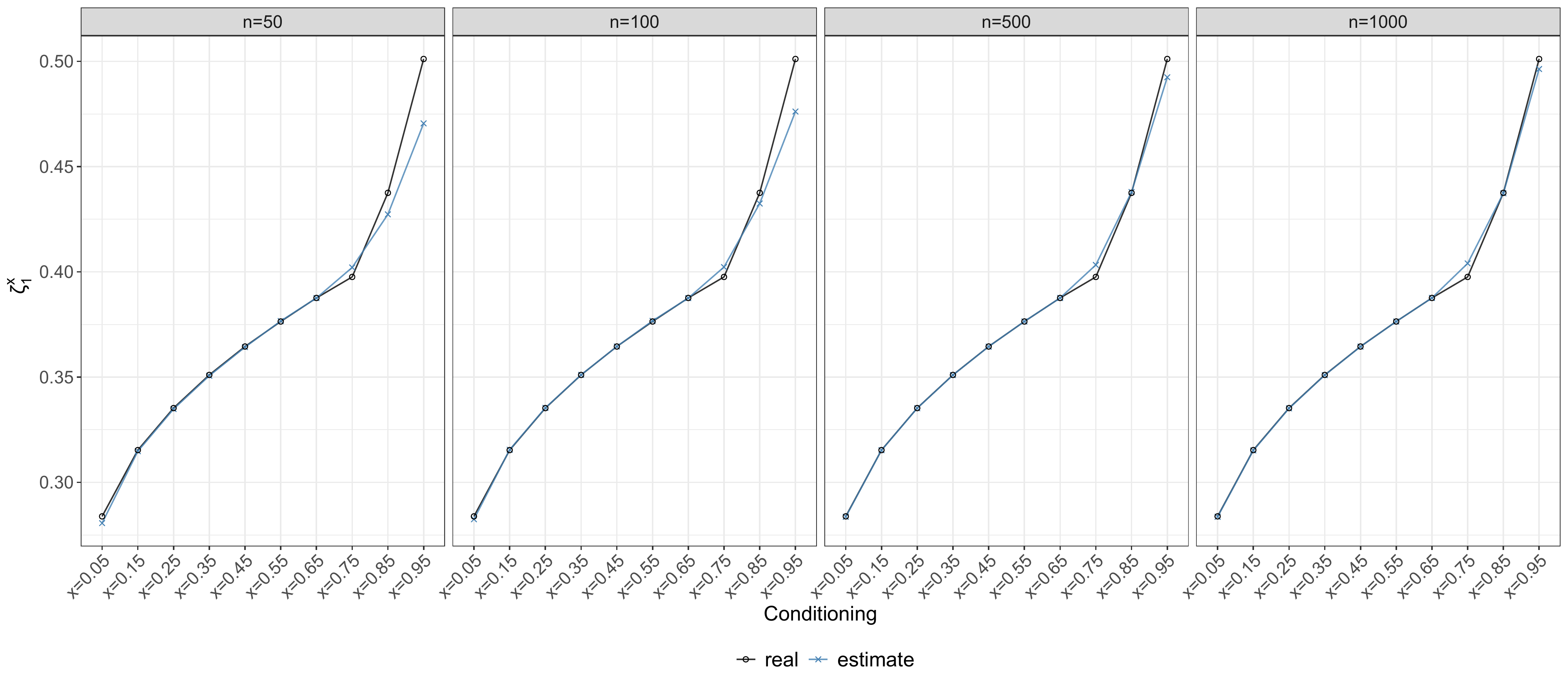}
			\caption{Values of $\zeta_1^x(C)$ and averages of $\zeta_1^x(\hat{C}_n)$ over $R=30=$ runs for $x\in\{0.05,0.15,\ldots,0.85\}$ and sample sizes $n=\{50,100,500,1000\}$ as considered in Example \ref{ex:est:zeta1x}.}
			\label{fig:ex:est:zeta1x}
		\end{figure} 
		The quality of the estimation stems from our choice of simulating data from a parametric family and using parametric estimation. The more general non-parametric estimation is more difficult even in dimension $3$ and left for future work on the topic.
	\end{exa}

	We conculde the paper by showing that, in addition to convergence, singularity as well as conditional increasingness of an Archimedean copula carry over to its conditional copulas, too.

	\begin{theorem}\label{th:singularity:conditional:copula}
		Suppose that $C$ is a $d$-dimensional Archimedean copula with strict generator $\psi$. If $C$ is singular w.r.t. $\lambda_d$ then for every  $\ell\in\{1,\ldots,d-2\}$ there exists a set $\Lambda\in\mathcal{B}(\mathbb{I}^\ell)$ with $\lambda_\ell(\Lambda) = 1$ such that for all $\mathbf{x}\in\Lambda$ the conditional copula $C^\mathbf{x}$ is singular w.r.t. $\lambda_{d-\ell}$.
	\end{theorem}
	\begin{proof}
		According to \cite[Section 3.2]{Principles} it suffices to show that the $(d-\ell)$-th derivative vanishes $\lambda_\ell$-almost everywhere.
		Straightforward calculations yield
		\begin{align*}
			&\frac{\partial^{d-\ell}}{\partial u_1 \cdots \partial u_{d-\ell}} C^\mathbf{x}(\mathbf{u}) = \tfrac{\psi^{(d)}\left( \sum_{i=1}^\ell \varphi(x_i) + \sum_{j=1}^{d-\ell} \varphi(g_\mathbf{x}^{-1}(u_j)) \right)}{\psi^{(\ell)}\left( \sum_{i=1}^\ell\varphi(x_i) \right)}  \prod_{j=1}^{d-\ell} \varphi'(g_\mathbf{x}^{-1}(u_j))  (g_\mathbf{x}^{-1})'(u_j) \\
			&= \underbrace{\tfrac{\psi^{(d)}\left( \sum_{i=1}^\ell\varphi(x_i) + \sum_{j=1}^{d-\ell}\varphi(u_j) \right)}{\psi^{(\ell)}\left( \sum_{i=1}^\ell\varphi(x_i) \right)}}_{= \frac{\partial^{d-\ell}}{\partial u_1\cdots \partial u_{d-\ell}}K_C(\mathbf{x},[\mathbf{0},\mathbf{u}])\prod_{j=1}^{d-\ell}\frac{1}{\varphi'(u_j)}}
			\cdot \tfrac{\psi^{(d)}\left( \sum_{i=1}^\ell \varphi(x_i) + \sum_{j=1}^{d-\ell} \varphi(g_\mathbf{x}^{-1}(u_j)) \right)}{\psi^{(d)}\left( \sum_{i=1}^\ell\varphi(x_i) + \sum_{j=1}^{d-\ell}\varphi(u_j) \right)} \prod_{j=1}^{d-\ell} \varphi'(g_\mathbf{x}^{-1}(u_j))  (g_\mathbf{x}^{-1})'(u_j).
		\end{align*}
		Since $C$ is singular, however, Theorem \ref{th:singular:kernels} yields a set $\Lambda\in\mathcal{B}(\mathbb{I}^\ell)$ with $\lambda_\ell(\Lambda) = 1$ such that for all $\mathbf{x}\in\Lambda$ we have that $K_C(\mathbf{x},\cdot)$ is singular w.r.t. $\lambda_{d-\ell}$. As direct consequence, $\frac{\partial^{d-\ell}}{\partial u_1\cdots \partial u_{d-\ell}}K_C(\mathbf{x},[\mathbf{0},\mathbf{u}]) = 0$ for all $\mathbf{x}\in\Lambda$ so the above expression evaluates to $0$ as well and singularity of $C^\mathbf{x}$ w.r.t. $\lambda_{d-\ell}$ follows.
	\end{proof}

	\vspace*{0.3cm}
	
	Following \cite{mueller_scarsini}, $C$ is conditionally increasing if the random vector $\mathbf{X}\sim C$ and $X_i$ is stochastically increasing in $(X_p, p\in J)$ for all $i\not\in J$ and $J\subset\{1,\ldots, d\}$. 
	Translating to the Markov kernel setup, $C$ is conditionally increasing if for every $\ell\in\{1,\ldots, d-1\}$ we have that $\mathbf{x}\mapsto K_C(\mathbf{x},[0,y])$ is decreasing in $\mathbf{x}$ for all $y\in\mathbb{I}$ (w.l.o.g., we take the first $\ell$ coordinates).

	\begin{Prop}
		Suppose that $C$ is a $d$-dimensional Archimedean copula with generator $\psi$ and let $\ell\in\{1,\ldots, d-2\}$. If $C$ is conditionally increasing, then so is $C^\mathbf{x}$ for any $\mathbf{x}\in(0,1)^\ell$.
		\begin{proof}
			According to \cite[Theorem 2.8]{mueller_scarsini} we have to prove that $(-1)^{d-\ell-1}(\psi^\mathbf{x})^{(d-\ell-1)}$ is log-convex and writing
			\begin{align*}
				(-1)^{d-\ell-1}(\psi^\mathbf{x})^{(d-\ell-1)}(z) &= \frac{(-1)^{d-1}D^-\psi^{(d-2)}\left( \sum_{i=1}^\ell\varphi(x_i) + z \right)}{(-1)^{\ell}\psi^{(\ell)}\left( \sum_{i=1}^\ell\varphi(x_i) \right)}
			\end{align*}
			we directly observe that log-convexity follows immediately from log-convexity of $(-1)^{d-1} D^-\psi^{(d-2)}$.
		\end{proof}
		%		\noindent Furthermore, in the case of $\ell=d-2$, we have $C^\mathbf{x}\in\mathcal{C}^2_\text{ar}$ and hence, $C^\mathbf{x}$ is also MK-TP2.
	\end{Prop}

	\section*{Acknowledgments}
	The author gratefully acknowledges the financial support from Porsche Holding Austria and 
	Land Salzburg within the WISS 2025 project \textquoteleft KFZ' (P1900123).
	
	\appendix
	
	%\section*{Appendix}
	\section{Multivariate conditional distributions}
	\label{sec:kernel:representation}
	This section derives several results concerning the connection of derivatives and Markov kernels of general multivariate distribution functions with absolutely continuous marginals which are applied to our setting of Archimedean copulas in the main text. 
	Suppose $H$ is a $d$-dimensional distribution function and let $\ell\in\{1,\ldots,d-1\}$. Considering the disintegration theorem, (a version of) the $\ell$-Markov kernel $K_H$ is given by the Radon-Nikodym derivative of $\mu_H$ w.r.t. the marginal measure $\mu_{H^{1:\ell}}$, i.e., $K_H = \frac{\mathrm{d}\mu_H}{\mathrm{d}\mu_{H^{1:\ell}}}$,
	whence deriving it is generally highly non-trivial.

	For some special cases, however, it is well-known that more can be said. For instance, if $H$ is absolutely continuous then there exists a density $h^{1:d}$ of $H$ and disintegration directly yields
	\begin{align*}
		H(\mathbf{x},\mathbf{y}) = \int_{(\mathbf{-\infty},\mathbf{x}]} K_H(\mathbf{s}, (\mathbf{-\infty},\mathbf{y}]) \ \mathrm{d}\mu_{H^{1:\ell}}(\mathbf{s}) = \int_{(\mathbf{-\infty},\mathbf{x}]} \int_{(\mathbf{-\infty},\mathbf{y}]} h^{1:d}(\mathbf{s},\mathbf{t}) \ \mathrm{d}\lambda_{d-\ell}(\mathbf{t})\mathrm{d}\lambda_\ell(\mathbf{s}).
	\end{align*}
	Considering that $\mu_{H^{1:\ell}}$ is also absolutely continuous, there exists $\Lambda\in\mathcal{B}(\mathbb{R}^\ell)$ with $\lambda_\ell(\Lambda) = 1$ such that for all $\mathbf{x}\in\Lambda$ we have
	\begin{align*}
		K_H(\mathbf{x},(\mathbf{-\infty}, \mathbf{y}]) = \int_{(\mathbf{-\infty},\mathbf{y}]} \frac{h^{1:d}(\mathbf{x},\mathbf{t})}{h^{1:\ell}(\mathbf{x})} \ \mathrm{d}\lambda_{d-\ell}(\mathbf{t}) = \int_{(\mathbf{-\infty},\mathbf{y}]} h^{1:d|1:\ell}(\mathbf{t}|\mathbf{x}) \ \mathrm{d}\lambda_{d-\ell}(\mathbf{t}),
	\end{align*}
	where $h^{1:d|1:\ell}(\cdot|\mathbf{x})$ is the conditional density of $H$ given $\mathbf{x}$ (cf. \cite[Section 8.3]{Klenke} for the two dimensional case).
	
	Furthermore, in dimension two it is well-known that for every $y\in\mathbb{R}$ there exists a set $\Gamma^y$ with $\lambda(\Gamma^y) = 1$ such that for every $x\in\Gamma^y$ it holds that
	\begin{align*}
		K_H(x,(-\infty,y]) = \frac{\partial H(x,y)}{\partial x}
	\end{align*}
	regardless of $H$ being purely absolutely continuous, singular or mixed (see, e.g. \cite[Theorem 3.4.4]{Principles} for a proof in the copula setting). The next theorem extends this assertion to multivariate disitribution functions with absolutely continuous marginals:
	
	\begin{theorem}\label{th:interkernel}
		Suppose that $H$ is a $d$-dimensional distribution function with absolutely continuous $\ell$-marginal for some $\ell\in\{2,\ldots,d-1\}$. Then for every $\mathbf{y}\in\mathbb{R}^{d-\ell}$ there exists $\Gamma^\mathbf{y}\in\mathcal{B}(\mathbb{R}^\ell)$ with $\lambda_\ell(\Gamma^\mathbf{y}) = 1$ such that
		\begin{align}
			K_H(\mathbf{x},(\mathbf{-\infty}, \mathbf{y}]) = \frac{\partial^\ell H(\mathbf{x},\mathbf{y})}{\partial x_1 \partial x_2 \cdots \partial x_\ell}  \cdot h^{1:\ell}(\mathbf{x})^{-1}
		\end{align}
		holds for all $\mathbf{x}\in\Gamma^\mathbf{y}$.
	\end{theorem}
	The proof of this theorem is essentially carried out in \cite[Supplementary]{qmd} where the authors show that for a $d$-dimensional copula $C$ with only independent $(d-1)$-marginals, i.e. $C^{1:d-1} = \Pi_{d-1}$, it holds that for every $y\in\mathbb{I}$ we have
	\begin{align*}
		K_C(x_1,x_2,\ldots x_{d-1}, [0,y]) =\frac{\partial^{d-1} C(\mathbf{x},y)}{\partial x_1 \partial x_2 \cdots \partial x_{d-1}}
	\end{align*}
	for $\lambda_{d-1}$-almost every $\mathbf{x}\in\mathbb{I}^{d-1}$. The proof can be readily adapted to our setting.

	Following \cite{Principles} a $d$-dimensional distribution function $G$ is $d$-increasing if for $\mathbf{a}, \mathbf{b}\in\mathbb{R}^d$ such that $b_i \geq a_i$ for $i=1,2,\ldots,d$ it holds that
	\begin{align*}
		V_{G}((\mathbf{a},\mathbf{b}]) := \sum_{\mathbf{v}\in\text{ver}((\mathbf{a},\mathbf{b}])}\text{sign}(\mathbf{v}) G(\mathbf{v}) \geq 0,
	\end{align*}
	where sign$(\mathbf{v}) = 1$ if $v_j=a_j$ for an even number of indices and sign$(\mathbf{v}) =-1$ otherwise and ver$((\mathbf{a},\mathbf{b}]) = \{a_1,b_1\} \times \ldots \times \{a_{d}, b_{d}\}$. Using the finite difference operator $\Delta_{a_i,b_i}^j$ given by
	\begin{align*}
		\Delta_{a_i,b_i}^j G(\mathbf{s}) := G(s_1,\ldots,s_{j-1},b_j,s_{j+1},\ldots, s_d) - G(s_1,\ldots,s_{j-1},a_j,s_{j+1},\ldots, s_d)
	\end{align*}
	allows to alternatively compute the $G$-volume $V_G$ of $(\mathbf{a},\mathbf{b}]$ via
	\begin{align*}
		V_{G}((\mathbf{a},\mathbf{b}]) = \Delta^d_{a_d,b_d}\Delta^{d-1}_{a_{d-1},b_{d-1}} \cdots \Delta^1_{a_i,b_1} G(\mathbf{s}).
	\end{align*}
	This expression will be useful in the next lemma saying that if it exists, the $\ell$-times iterated partial derivatives of a $d$-dimensional distribution function is $(d-\ell)$-increasing.
	
	\begin{Lemma}\label{lem:partial:increasing}
		Suppose that $H$ is a $d$-dimensional distribution function and $\ell\in\{1,\ldots,d-1\}$. If the $\ell$-th order partial derivatives exist and are continuous then the function
		\begin{align}\label{eq:lth:derivative}
			F_\mathbf{x}(\mathbf{y}) = \frac{\partial^\ell H(\mathbf{x},\mathbf{y})}{\partial x_1 \partial x_2 \cdots \partial x_\ell}
		\end{align}
		is $(d-\ell)$-increasing.
	\end{Lemma}
	\begin{proof}
		Applying the notation $\mathbf{s}_{i:j} = (s_i, s_{i+1},\ldots, s_j)$ for $i,j\in\mathbb{N}$ with $i<j$ we fix $\mathbf{a}_{\ell+1:d}, \mathbf{b}_{\ell+1:d}$ such that $b_i\geq a_i$ for $i\in\{\ell+1,\ldots, d\}$. 
		
		Using the expression in terms of the finite difference operator and its relation to partial derivatives we get
		\begin{align*}
			V_{F_\mathbf{x}}((\mathbf{a}_{\ell+1:d},\mathbf{b}_{\ell+1:d}]) &= \Delta^d_{a_d,b_d} \cdots \Delta^{\ell+1}_{a_{\ell+1},b_{\ell+1}} F_\mathbf{x}(\mathbf{s}) \\
			&= \Delta^d_{a_d,b_d} \cdots \Delta^{\ell+1}_{a_{\ell+1},b_{\ell+1}} \lim_{h\downarrow 0} \frac{1}{h^\ell} \Delta^\ell_{x_\ell + h, x_\ell} \cdots \Delta^1_{x_1+h,x_1} H(\mathbf{x},\mathbf{s}) \\
			&= \lim_{h\downarrow 0} \frac{1}{h^\ell} \underbrace{V_H((\mathbf{a},\mathbf{b}])}_{\geq 0} \geq 0,
		\end{align*}
		where for $h>0$ we set $\mathbf{a} = (x_1,\ldots,x_\ell, a_{\ell+1},a_{\ell+2},\ldots, a_d)$ and $\mathbf{b} = (x_1+h,\ldots,x_\ell+h, b_{\ell+1},b_{\ell+2},\ldots, b_d)$.	
	\end{proof}
	
	\section{Auxiliary results for multivariate copulas}\label{sec:aux:copulas}
	
	Instead of general multivariate distributions, this sections considers $d$-dimensional copulas and provides useful lemmata which are relevant for our main results.

	\begin{Lemma}\label{lem:abs:cts:kernel}
		Suppose $C$ is a $d$-dimensional copula and denote by $c^{1:d}$ ($c^{1:\ell}$) the Radon-Nikodym derivative of $\mu_C$ w.r.t. $\lambda_d$ ($\mu_{C^{1:\ell}}$ w.r.t. $\lambda_\ell$). Then there exists a set $\Lambda\in\mathcal{B}(\mathbb{I}^\ell)$ with $\lambda_\ell(\Lambda) = 1$ such that for all $\mathbf{x}\in\Lambda$ we have
		\begin{align*}
			K_C^\text{abs}(\mathbf{x}, F) = \int_F \frac{c^{1:d}(\mathbf{x},\mathbf{y})}{c^{1:\ell}(\mathbf{x})} \ \mathrm{d}\lambda_{d-\ell}(\mathbf{y}).
		\end{align*}
	\end{Lemma}
	\begin{proof}
		Using Lebesgue decomposition of $\mu_C$, disintegration and absolute continuity, on the one hand we have
		\begin{align}\label{eq:abs:kernel:1}
			\mu_C(E\times F) &= \mu_C^\text{abs}(E\times F) + \mu_C^\text{sing}(E\times F) \\
			&= \int_E\int_F c^{1:d}(\mathbf{x},\mathbf{y}) \ \mathrm{d}\lambda_{d-\ell}(\mathbf{y})\mathrm{d}\lambda_\ell(\mathbf{x}) + \int_E K_{C^\text{sing}}(\mathbf{x},F) \ \mathrm{d}(\mu_C^\text{sing})^{1:\ell} \nonumber
		\end{align}
		and, on the other hand, 
		\begin{align}\label{eq:abs:kernel:2}
			\mu_C(E\times F) &= \int_E K_C(\mathbf{x}, F) \ \mathrm{d}\mu_{C^{1:\ell}}(\mathbf{x}) \\
			&= \int_E K_C^\text{abs}(\mathbf{x}, F) \ \mathrm{d}\mu_{C^{1:\ell}}(\mathbf{x}) + \int_E K_C^\text{sing}(\mathbf{x}, F) \ \mathrm{d}\mu_{C^{1:\ell}}(\mathbf{x}) \nonumber  \\
			&= \int_E K_C^\text{abs}(\mathbf{x}, F) \ \mathrm{d}(\mu_{C^{1:\ell}})^\text{abs}(\mathbf{x}) + \int_E K_C^\text{abs}(\mathbf{x}, F) \ \mathrm{d}(\mu_{C^{1:\ell}})^\text{sing}(\mathbf{x}) \nonumber \\ 
			&\qquad + \int_E K_C^\text{sing}(\mathbf{x}, F) \ \mathrm{d}(\mu_{C^{1:\ell}})^\text{abs}(\mathbf{x}) + \int_E K_C^\text{sing}(\mathbf{x}, F) \ \mathrm{d}(\mu_{C^{1:\ell}})^\text{sing}(\mathbf{x}) \nonumber\\
			&= \underbrace{\int_E K_C^\text{abs}(\mathbf{x}, F) c^{1:\ell}(\mathbf{x}) \ \mathrm{d}\lambda_\ell(\mathbf{x})}_{(a)} + \underbrace{\int_E K_C^\text{abs}(\mathbf{x}, F) \ \mathrm{d}(\mu_{C^{1:\ell}})^\text{sing}(\mathbf{x})}_{(b)} \nonumber\\ 
			&\qquad + \underbrace{\int_E K_C^\text{sing}(\mathbf{x}, F)c^{1:\ell}(\mathbf{x}) \ \mathrm{d}\lambda_\ell(\mathbf{x})}_{(c)} + \underbrace{\int_E K_C^\text{sing}(\mathbf{x}, F) \ \mathrm{d}(\mu_{C^{1:\ell}})^\text{sing}(\mathbf{x})}_{(d)} \nonumber
			%						&= \int_E K_C^\text{abs}(\mathbf{x}, F) c^{1:\ell}(\mathbf{x}) + K_C^\text{sing}(\mathbf{x},F)c^{1:\ell}(\mathbf{x}) \ \mathrm{d}\lambda_\ell(\mathbf{x}) \\
			%						&\qquad + \int_E K_C^\text{abs}(\mathbf{x},F) + K_C^\text{sing}(\mathbf{x}, F) \ \mathrm{d}\mu_{C^{1:\ell}}^\text{sing}
		\end{align}
		via disintegration, decomposition of the Markov kernel and then Lebesgue decomposition of the marginal measure $\mu_{C^{1:\ell}}$. Notice additionally that we have both
		\begin{align*}
			\mu_C(E\times\mathbb{I}^{d-\ell}) &= \mu_{C^{1:\ell}}(E) = \mu_{C^{1:\ell}}^\text{abs}(E) + \mu_{C^{1:\ell}}^\text{sing}(E), \\
			\mu_C(E\times\mathbb{I}^{d-\ell}) &= \mu_C^\text{abs}(E\times\mathbb{I}^{d-\ell}) + \mu_C^\text{sing}(E\times\mathbb{I}^{d-\ell}) = (\mu_C^\text{abs})^{1:\ell}(E) + (\mu_C^\text{sing})^{1:\ell}(E)
		\end{align*}
		and therefore $\mu_{C^{1:\ell}}^\text{abs}(E) = (\mu_C^\text{abs})^{1:\ell}(E)$ as well as $\mu_{C^{1:\ell}}^\text{sing}(E) = (\mu_C^\text{sing})^{1:\ell}(E)$. \\
		(a): The measure $\Gamma_\text{a}(E\times F) := \int_E K_C^\text{abs}(\mathbf{x}, F) c^{1:\ell}(\mathbf{x}) \ \mathrm{d}\lambda_\ell(\mathbf{x})$ is absolutely continuous w.r.t. $\lambda_{d-\ell}$ as every $\lambda_{d-\ell}$-null set is also a $\Gamma_\text{a}$-null set. \\ (b): Defining $\Gamma_\text{b}(E\times F) := \int_E K_C^\text{abs}(\mathbf{x}, F) \ \mathrm{d}(\mu_{C^{1:\ell}})^\text{sing}(\mathbf{x})$ there exists a set $S\in\mathcal{B}(\mathbb{I}^\ell)$ such that $\mu_{C^{1:\ell}}^\text{sing}(S) = 1$ while $\lambda_\ell(S)=0$. Considering that $K_C^\text{abs}$ is a subkernel (i.e., $K_C^\text{abs}(\mathbf{x}, \mathbb{I}^{d-\ell}) \leq 1$) we have
		\begin{align*}
			\Gamma_\text{b}(\mathbb{I}^d) &= \int_{\mathbb{I}^\ell} K_C^\text{abs}(\mathbf{x}, \mathbb{I}^{d-\ell}) \ \mathrm{d}\mu_{C^{1:\ell}}^\text{sing}(\mathbf{x}) = \int_{S} K_C^\text{abs}(\mathbf{x}, \mathbb{I}^{d-\ell}) \ \mathrm{d}\mu_{C^{1:\ell}}^\text{sing}(\mathbf{x}) = \parallel \Gamma_\text{b}\parallel = \Gamma_\text{b}(S\times\mathbb{I}^{d-\ell})
		\end{align*}
		whereas
		\begin{align*}
			\lambda_d(S\times\mathbb{I}^{d-\ell}) = \int_S \lambda_{d-\ell}(\mathbb{I}^{d-\ell}) \ \mathrm{d}\lambda_\ell = 0.
		\end{align*}
		The same argument yields that the measure defined according to (d) is singular w.r.t. $\lambda_{d}$, too. \\ 
		For part (c) we start by defining $\Gamma_\text{c}(E\times F) := \int_E K_C^\text{sing}(\mathbf{x}, F) c^{1:\ell}(\mathbf{x}) \ \mathrm{d}\lambda_\ell(\mathbf{x})$. Observe that the marginal $\Gamma_\text{c}^{1:\ell}$ is absolutely continuous w.r.t. $\lambda_\ell$. In fact,
		\begin{align*}
			\Gamma_\text{c}^{1:\ell}(E) = \Gamma_\text{c}(E\times\mathbb{I}^{d-\ell}) = \int_E K_C^\text{sing}(\mathbf{x}, \mathbb{I}^{d-\ell}) c^{1:\ell}(\mathbf{x}) \ \mathrm{d}\lambda_\ell(\mathbf{x})
		\end{align*}
		so disintegration yields
		\begin{align*}
			\Gamma_\text{c}(E\times F) = \int_E K_{\Gamma_\text{c}}(\mathbf{x}, F) \ \mathrm{d}\Gamma_\text{c}^{1:\ell}(\mathbf{x}) = \int_E  K_{\Gamma_\text{c}}(\mathbf{x}, F) K_C^\text{sing}(\mathbf{x},\mathbb{I}^{d-\ell}) c^{1:\ell}(\mathbf{x}) \ \mathrm{d}\lambda_\ell(\mathbf{x}).
		\end{align*}
		Consequently, there exists a set $\Upsilon_1\in\mathcal{B}(\mathbb{I}^\ell)$ with $\lambda_\ell(\Upsilon_1) = 1$ such that for all $\mathbf{x}\in\Upsilon_1$ we have
		\begin{align*}
			K_C^\text{sing}(\mathbf{x}, F) c^{1:\ell}(\mathbf{x}) = K_{\Gamma_\text{c}}(\mathbf{x}, F) K_C^\text{sing}(\mathbf{x},\mathbb{I}^{d-\ell}) c^{1:\ell}(\mathbf{x})
		\end{align*}
		and hence
		\begin{align*}
			K_{\Gamma_\text{c}}(\mathbf{x}, F) = \frac{K_C^\text{sing}(\mathbf{x}, F)}{K_C^\text{sing}(\mathbf{x},\mathbb{I}^{d-\ell})}
		\end{align*}
		which is the normalization of the subkernel. It directly follows that $K_{\Gamma_\text{c}}(\mathbf{x},\cdot)$ is singular w.r.t. $\lambda_{d-\ell}$ for all $\mathbf{x}\in\Upsilon_1$. We now proceed as before and decompose $\Gamma_\text{c}$ twice:
		\begin{align*}
			\Gamma_\text{c}(E\times F) &= \Gamma_\text{c}^\text{abs}(E\times F) + \Gamma_\text{c}^\text{sing}(E\times F) \\
			&= \int_E \int_F g^{1:d}(\mathbf{x},\mathbf{y}) \ \mathrm{d}\lambda_{d-\ell}(\mathbf{y})\mathrm{d}\lambda_\ell(\mathbf{x}) + \int_E K_{\Gamma_\text{c}^\text{sing}}(\mathbf{x}, F) \ \mathrm{d}\Gamma_\text{c}^{1:\ell}(\mathbf{x}) \\
			&= \int_E \int_F g^{1:d}(\mathbf{x},\mathbf{y}) \ \mathrm{d}\lambda_{d-\ell}(\mathbf{y})\mathrm{d}\lambda_\ell(\mathbf{x}) \\
			&\qquad + \int_E K_{\Gamma_\text{c}^\text{sing}}(\mathbf{x}, F)  K_C^\text{sing}(\mathbf{x},\mathbb{I}^{d-\ell})c^{1:\ell}(\mathbf{x}) \ \mathrm{d}\lambda_\ell(\mathbf{x}),
		\end{align*}
		where $g^{1:d}$ denotes the Radon-Nikodym density of $\Gamma_\text{c}^\text{abs}$ w.r.t. $\lambda_d$. Additionally,
		\begin{align*}
			\Gamma_\text{c}(E\times F) &= \int_E K_{\Gamma_\text{c}}(\mathbf{x}, F) \ \mathrm{d}\Gamma_\text{c}^{1:\ell}(\mathbf{x}) \\
			&= \int_E \frac{K_C^\text{sing}(\mathbf{x}, F)}{K_C^\text{sing}(\mathbf{x},\mathbb{I}^{d-\ell})} \cdot K_C^\text{sing}(\mathbf{x},\mathbb{I}^{d-\ell})c^{1:\ell}(\mathbf{x}) \ \mathrm{d}\lambda_\ell(\mathbf{x}).
		\end{align*}
		Consequently, there exists a set $\Upsilon_2\in\mathcal{B}(\mathbb{I}^\ell)$ with $\lambda_\ell(\Upsilon_2) = 1$ such that for all $\mathbf{x}\in\Upsilon_2$ we have
		\begin{align*}
			& K_C^\text{sing}(\mathbf{x}, F)c^{1:\ell}(\mathbf{x}) = \int_F g^{1:d}(\mathbf{x},\mathbf{y}) \ \mathrm{d}\lambda_{d-\ell}(\mathbf{y}) + K_{\Gamma_\text{c}^\text{sing}}(\mathbf{x}, F) \cdot K_C^\text{sing}(\mathbf{x},\mathbb{I}^{d-\ell})c^{1:\ell}(\mathbf{x}) 	
		\end{align*}
		and considering that $\{\mathbf{x}\in\mathbb{I}^\ell: c^{1:\ell}(\mathbf{x})>0\}$ is of full $\mu_{C^{1:\ell}}$-measure and by absolute continuity of full $\lambda_\ell$-measure too, yields the existence of a set $\Upsilon_3\in\mathcal{B}(\mathbb{I}^\ell)$ with $\lambda_\ell(\Upsilon_3) = 1$ such that for all $\mathbf{x}\in\Upsilon_3$ we have
		\begin{align*}
			\frac{K_C^\text{sing}(\mathbf{x}, F)}{K_C^\text{sing}(\mathbf{x},\mathbb{I}^{d-\ell})} = \int_F \frac{g^{1:d}(\mathbf{x},\mathbf{y})}{c^{1:\ell}(\mathbf{x})K_C^\text{sing}(\mathbf{x},\mathbb{I}^{d-\ell})} \ \mathrm{d}\lambda_{d-\ell}(\mathbf{y}) + K_{\Gamma_\text{c}^\text{sing}}(\mathbf{x}, F).
		\end{align*}
		However, we already know that $K_{\Gamma_\text{c}}^\text{sing}(\mathbf{x}, F) = K_{\Gamma_\text{c}}(\mathbf{x},F) = \frac{K_C^\text{sing}(\mathbf{x}, F)}{K_C^\text{sing}(\mathbf{x},\mathbb{I}^{d-\ell})}$ so that
		\begin{align*}
			\int_F \frac{g^{1:d}(\mathbf{x},\mathbf{y})}{c^{1:\ell}(\mathbf{x})K_C^\text{sing}(\mathbf{x},\mathbb{I}^{d-\ell})} \ \mathrm{d}\lambda_{d-\ell}(\mathbf{y}) = 0
		\end{align*}
		and therefore $\Gamma_\text{c}^\text{abs}(E\times F) = 0$ holds which implies that $\Gamma_\text{c}$ is purely singular. \\
		Having this we can finally compare the two absolutely continuous components of eq. \eqref{eq:abs:kernel:1} and eq. \eqref{eq:abs:kernel:2} so that for $\mu_C(E\times F)$ we get the existence of a set $\Lambda\in\mathcal{B}(\mathbb{I}^\ell)$ with $\lambda_\ell(\Lambda) = 1$ such that for all $\mathbf{x}\in\Lambda$ we have
		\begin{align*}
			\int_F c^{1:d}(\mathbf{x},\mathbf{y}) \ \mathrm{d}\lambda_{d-\ell}(\mathbf{y}) = K_C^\text{abs}(\mathbf{x}, F) c^{1:\ell}(\mathbf{x})
		\end{align*}
		from which the assertion follows.
	\end{proof}
	
	The subsequent lemma concerns the approximation of $\zeta_1^d$. Recall that the total variation metric of two probability measures $\mu,\nu$ on $(\Omega, \mathcal{A})$ is given by $d_{TV}(\mu, \nu) = \sup_{A\in\mathcal{A}}|\mu(A) - \nu(A)|$ and that it is well-known that for absolutely continuous probability measures, convergence w.r.t. $d_{TV}$ is equivalent to a.e.-convergence of the corresponding densities (see, e.g., \cite{total_var_lecture_notes}).
	\begin{Lemma}\label{lem:conv:cond:zeta}
		Suppose $C,C_1,C_2,\ldots$ are $d$-dimensional copulas with $(d-1)$-Markov kernels $K_C, K_{C_1}, K_{C_2},\ldots$, respectively. If $(C_n)_{n\in\mathbb{N}}$ converges $(d-1)$-weakly conditional to $C$ and $(\mu_{C_n^{1:d-1}})_{n\in\mathbb{N}}$ converges in total variation to $\mu_{C^{1:d-1}}$ then
		\begin{align*}
			\lim_{n\to\infty}\zeta_1^d(C_n) = \zeta_1^d(C).
		\end{align*}
	\end{Lemma}
	\begin{proof}
		Defining $\beta_n(\mathbf{x},y) := K_{C_n}(\mathbf{x},[0,y])-y$, $\mathbf{x}\in\mathbb{I}^{d-1}, y\in\mathbb{I}$ (and $\beta$ analogously), we get
		\begin{align*}
			&\left|  \int_\mathbb{I}\int_{\mathbb{I}^{d-1}} \left| \beta_n \right| \ \mathrm{d}\mu_{C^{1:d-1}}\mathrm{d}\lambda - \int_\mathbb{I}\int_{\mathbb{I}^{d-1}} \left| \beta \right| \ \mathrm{d}\mu_{C^{1:d-1}}\mathrm{d}\lambda \right| \\
			&\leq \int_\mathbb{I} \left| \int_{\mathbb{I}^{d-1}} |\beta_n| \ \mathrm{d}\mu_{C_n^{1:d-1}} - \int_{\mathbb{I}^{d-1}} |\beta_n| \ \mathrm{d}\mu_{C^{1:d-1}} \right| \mathrm{d}\lambda\\
			& \qquad + \int_\mathbb{I} \left| \int_{\mathbb{I}^{d-1}} |\beta_n| \ \mathrm{d}\mu_{C^{1:d-1}} - \int_{\mathbb{I}^{d-1}} |\beta| \ \mathrm{d}\mu_{C^{1:d-1}} \right| \mathrm{d}\lambda
		\end{align*}
		For the first summand (I) we use the well-known fact (see, e.g., \cite{total_var_lecture_notes}) that $d_{TV}(C_n^{1:d-1}, C^{1:d-1})\xrightarrow{n\to\infty} 0$ if, and only if, 
		\begin{align*}
			\lim_{n\to\infty} \frac{1}{2}\sup_{\|g\|_\infty\leq 1} \left| \int_{\mathbb{I}^{d-1}} g \ \mathrm{d}\mu_{C_n^{1:d-1}} - \int_{\mathbb{I}^{d-1}} g \ \mathrm{d}\mu_{C^{1:d-1}} \right| = 0
		\end{align*}
		and estimate
		\begin{align*}
			(I) \leq \int_\mathbb{I} 2\cdot\frac{1}{2} \sup_{\|g\|_\infty\leq 1} \left| \int_{\mathbb{I}^{d-1}} g \ \mathrm{d}\mu_{C_n^{1:d-1}} - \int_{\mathbb{I}^{d-1}} g \ \mathrm{d}\mu_{C^{1:d-1}} \right| \xrightarrow{n\to\infty} 0.
		\end{align*}
		For the second summand (II) we compute
		\begin{align*}
			(II) &= \int_\mathbb{I} \left| \int_{\mathbb{I}^{d-1}} |\beta_n| - |\beta| \ \mathrm{d}\mu_{C^{1:d-1}} \right| \ \mathrm{d}\lambda \leq \int_\mathbb{I} \int_{\mathbb{I}^{d-1}} \left| |\beta_n| - |\beta| \right| \ \mathrm{d}\mu_{C^{1:d-1}} \mathrm{d}\lambda \\
			&\leq \int_\mathbb{I} \int_{\mathbb{I}^{d-1}} \left| \beta_n - \beta \right| \ \mathrm{d}\mu_{C^{1:d-1}} \mathrm{d}\lambda = \int_{\mathbb{I}^{d-1}} \int_\mathbb{I}\left| \beta_n - \beta \right| \ \mathrm{d}\lambda\mathrm{d}\mu_{C^{1:d-1}} \\
			&= \int_{\mathbb{I}^{d-1}} \int_\mathbb{I}\left| K_{C_n}(\mathbf{x},[0,y]) - K_C(\mathbf{x},[0,y]) \right| \ \mathrm{d}\lambda(y)\mathrm{d}\mu_{C^{1:d-1}}(\mathbf{x}).
		\end{align*}
		By $(d-1)$-weak conditional convergence there exists $\Lambda\in\mathcal{B}(\mathbb{I}^{d-1})$ with $\mu_{C^{1:d-1}}(\Lambda) = 1$ such that for all $\mathbf{x}\in\Lambda$ we have that $
		|K_{C_n}(\mathbf{x},[0,y]) - K_C(\mathbf{x},[0,y])| \to 0
		$ as $n\to\infty$ for every continuity point $y$ of $K_C(\mathbf{x},\cdot)$ so using Dominated Convergence it follows that (II) $\to 0$ as $n\to\infty$ which completes the proof.
	\end{proof}

%	\section*{Acknowledgements}\vspace*{-0.2cm}
%	The author gratefully acknowledges the financial support from Porsche Holding Austria and 
%	Land Salzburg within the WISS 2025 project \textquoteleft KFZ' (P1900123).
%	The second author gratefully acknowledges the support of the WISS 2025 project \textquoteleft 
%	IDA-lab Salzburg' (20204-WISS/225/197-2019 and 20102-F1901166-KZP).

	\newpage
%	\bibliography{literature} 
	\bibliographystyle{Abbrv}

\end{document}